\documentclass[reqno,11pt]{amsart}
\usepackage{amssymb, amsmath,latexsym,amsfonts,amsbsy, amsthm}
\usepackage{color}
\usepackage{tikz}
\usetikzlibrary{positioning}

\allowdisplaybreaks[4]

\newcommand{\beq}{\begin{equation}}
	\newcommand{\eeq}{\end{equation}}
\newcommand{\ben}{\begin{eqnarray}}
	\newcommand{\een}{\end{eqnarray}}
\newcommand{\beno}{\begin{eqnarray*}}
	\newcommand{\eeno}{\end{eqnarray*}}

\newtheorem{remark}{Remark}[section]

\renewcommand{\theequation}{\thesection.\arabic{equation}}

\newtheorem{theorem}{Theorem}[section]

\newtheorem{lemma}[theorem]{Lemma}
\newtheorem{proposition}[theorem]{Proposition}

\numberwithin{equation}{section} 
\newtheorem{Theorem}{Theorem}[section]

\newtheorem{Remark}[Theorem]{Remark}

\begin{document}

\title[Vanishing viscosity limit in a strip]
	{Steady symmetric Navier-Stokes flow near Poiseuille flow: existence and vanishing viscosity limit}
	\date{\today}

\author{Zeyu Jia}
\address{School of  Mathematics, Shandong University, jinan, Shandong, 250100, China}
\email{jzy280012@126.com}

\author{Jia Song}
\address{School of  Mathematics, Shandong University, jinan, Shandong, 250100, China}
\email{jiamaths@mail.sdu.edu.cn}

\author{Tao Tao}
\address{School of  Mathematics, Shandong University, jinan, Shandong, 250100, China}
\email{taotao@sdu.edu.cn}

	\date{\today}
\maketitle

\renewcommand{\theequation}{\thesection.\arabic{equation}}
\setcounter{equation}{0}
\begin{abstract}
    We study the existence and vanishing viscosity limit of steady incompressible Navier–Stokes flows in a horizontally periodic strip with no-slip boundary conditions and midline-symmetric forcing of order $\epsilon$, where $\epsilon>0$ denotes the viscosity. Taking a Poiseuille flow as the leading-order Euler profile, we construct higher-order symmetric approximate solutions and establish the existence of a symmetric solution to the associated nonlinear error equation. Under suitable compatibility conditions on the forcing, this construction yields a family of midline-symmetric steady Navier–Stokes solutions that converges to the prescribed Poiseuille flow in $L^\infty$ at the optimal rate $O(\epsilon)$ as $\epsilon\to 0$.
\end{abstract}

\numberwithin{equation}{section}

\indent

\section{Introduction}
\indent 
The vanishing viscosity limit of the steady incompressible
Navier--Stokes equations is a fundamental problem in the mathematical
theory of fluid mechanics. It concerns the approximation of viscous
flows by solutions of the Euler equations as the viscosity tends to
zero. In domains with solid boundaries, this limit is particularly
delicate because the no-slip boundary condition for the Navier--Stokes
equations generally differs from the impermeability condition imposed
on the Euler equations. This mismatch may generate boundary layers,
so that simply neglecting the viscous term does not necessarily
provide a valid approximation near the boundary. A rigorous
justification therefore requires both a suitable construction of
approximate solutions and quantitative control of the corresponding
remainders; see, for instance, \cite{FGLT1, FGLT2, FPZ, GZ,GZ2,GN,GI1,GI2, Iyer1, Iyer2, Iyer3, IM1, IM2, Iyer4}. 
For results on the Sobolev stability of certain classes of shear flows
of Prandtl type, we refer to \cite{CWZ,GMM}. 

In the stationary setting, the inviscid limit leads to a natural
construction problem: given a steady Euler flow, can one construct
steady Navier--Stokes solutions that satisfy the prescribed boundary
conditions and converge to that Euler flow as the viscosity vanishes?
This problem is closely related to the selection of steady Euler
flows through the vanishing viscosity limit, a subject with classical
contributions by Prandtl \cite{prandtl} and Batchelor \cite{B}.
Further developments can be found in
\cite{edwards,feymann-lagerstrom,GYX,kim-formula,riley81,van wijngarden,W}.
The objective is not merely to establish the existence of viscous
solutions for each fixed positive viscosity, but to construct a
family of solutions whose asymptotic behavior can be controlled
as the viscosity tends to zero.

In this paper, we address this problem in a horizontally periodic
strip with no-slip boundary conditions and an external force that
is symmetric with respect to the midline. We choose a Poiseuille
flow as the leading-order Euler profile and construct a family of
symmetric steady Navier--Stokes solutions converging to this profile.
The external force is of the same order as the viscosity, and our
main result gives a uniform convergence rate of order $\epsilon$,
where $\epsilon>0$ denotes the viscosity.

More precisely, we consider the forced steady Navier-Stokes equations in a horizontally periodic strip $ \Omega = \mathbb{T} \times [0, 1] $ 
\begin{align} \label{NS equation}
\left \{
\begin{aligned}
    &u^{\epsilon } \partial_{x} u^{\epsilon } + v^{\epsilon } \partial_{y} u^{\epsilon } + \partial_{x} p^{\epsilon }-\epsilon\triangle u^{\epsilon } =\epsilon F_{u}, \\[5pt]
    &u^{\epsilon } \partial_{x} v^{\epsilon } + v^{\epsilon } \partial_{y} v^{\epsilon } + \partial_{y} p^{\epsilon }-\epsilon\triangle v^{\epsilon } =\epsilon F_{v}, \\[5pt]
    &\partial_{x} u^{\epsilon } + \partial_{y} v^{\epsilon } =0, \\[5pt]
\end{aligned}
\right .\
\end{align}
with no-slip  boundary condition
\begin{equation} \label{D-BC}
\begin{aligned}
    (u^{\epsilon }, v^{\epsilon })|_{y=1} = (0, 0),~~\forall  x\in[0, 2\pi ], \\
    (u^{\epsilon }, v^{\epsilon })|_{y=0} = (0, 0), ~~\forall  x\in[0, 2\pi ]. 
\end{aligned}
\end{equation}
Here $\epsilon$ is the viscosity, $(u^{\epsilon}, v^{\epsilon})$ is the velocity, $p^{\epsilon}$ is the pressure, $(F_u, F_v)$ are external force. We assume that the force $(F_u, F_v)$ are smooth, compactly supported in $ \Omega$,  and satisfy the following  conditions
\begin{align} \label{compatibility}
\begin{aligned}
    F_u (x, y)=F_u(x, 1-y),~~F_v(x, y) = -F_v(x, 1-y),~~ \int_0^{2\pi}F_u(x,y)dx=0, \forall y\in [0,1].
\end{aligned}
\end{align}
Thus, the horizontal component of the force is even under reflection
about $y=\frac12$, whereas the vertical component is odd.

Formally letting $\epsilon\to0$ in \eqref{NS equation} yields
the two-dimensional steady Euler equations
\begin{align} \label{steady Eule equation}
\left \{
\begin{aligned}
    &u_{e}\partial_{x} u_{e} + v_{e}\partial_{y}u_{e} + \partial_{x}p_{e} = 0, \\
    &u_{e}\partial_{x} v_{e} + v_{e}\partial_{y}v_{e} + \partial_{y}p_{e} = 0, \\
    &\partial_{x} u_{e} + \partial_{y} v_{e} = 0, \\
    &v_{e}|_{y=0,1} = 0, ~~ \forall x \in [0, 2\pi].
\end{aligned}
\right .\
\end{align}
We take the leading-order Euler flow to be the Poiseuille profile
\[
    \boldsymbol U_e=(u_e(y),0),
    \qquad
    u_e(y)=\alpha y(1-y),
\]
where $\alpha>0$ is fixed. Since $\boldsymbol U_e$ is independent
of the horizontal variable and has no vertical component,
\[
    \nabla\cdot\boldsymbol U_e=0,
    \qquad
    (\boldsymbol U_e\cdot\nabla)\boldsymbol U_e=0.
\]
Consequently, it solves \eqref{steady Eule equation} with constant
pressure. Moreover,
\[
    u_e(0)=u_e(1)=0,
    \qquad
    u_e(y)=u_e(1-y),
\]
so the reference flow satisfies both the no-slip boundary condition
and the prescribed midline symmetry.

An important feature of this choice is the absence of an order-one
mismatch between the leading tangential Euler velocity and the
stationary wall velocity. Nevertheless, the small-viscosity problem remains singular. For a perturbation \(\boldsymbol w=(w_1,w_2)\), linearization of the convection term around \(\boldsymbol U_e\) gives
$$ (\boldsymbol U_e\cdot\nabla)\boldsymbol w + (\boldsymbol w\cdot\nabla)\boldsymbol U_e = u_e(y)\partial_x\boldsymbol w + w_2u_e'(y)\boldsymbol e_1, \qquad \boldsymbol e_1=(1,0). $$
The coefficient of horizontal transport vanishes at both walls,
while the regularizing effect of the viscous term weakens as
$\epsilon\to0$. Thus, estimates sufficient to solve the equations
for a fixed positive viscosity do not automatically provide
the control required to justify the inviscid limit. In addition,
although the external force disappears from the formal Euler
equations, it enters at the same order as the viscous term and
must be taken into account in the construction of the corrections.\\

Our main result can be stated as follows.
\begin{theorem} \label{total theorem}
   Assume that the force profiles $F_u$ and $F_v$ satisfy the
regularity and support assumptions stated above, together with
the conditions \eqref{compatibility}. Then there exist
constants $\epsilon_0>0$ and $C>0$ such that, for every
$\epsilon\in(0,\epsilon_0]$, the Navier--Stokes system
\eqref{NS equation} with boundary condition \eqref{D-BC}
admits a solution $(u^\epsilon,v^\epsilon)$ satisfying 
   \begin{align}
    u^\epsilon(x, y)=u^\epsilon(x, 1-y),~~v^\epsilon(x, y) = -v^\epsilon(x, 1-y).\nonumber
\end{align} 
Moreover, there holds 
   \[
    \|u^{\epsilon}(x, y)-u_e(y)\|_{L^{\infty}}+\|v^{\epsilon}(x, y)\|_{L^{\infty}} \le C\epsilon,
    \]
    where $C$ is independent of $\epsilon$.
\end{theorem}

\begin{remark}[Periodicity of the pressure]\label{rem:pressure}
The horizontal periodicity in Theorem~\ref{total theorem}
is imposed on the velocity field and the pressure gradient,
but not on the pressure itself. More precisely, the pressure
is understood on the covering strip $\mathbb{R}\times(0,1)$
and has the form
\[
    p^\epsilon(x,y)
    =
    \widetilde p^\epsilon(x,y)+\Pi_\epsilon x,
    \qquad
    \widetilde p^\epsilon(x+2\pi,y)
    =
    \widetilde p^\epsilon(x,y),
\]
where $\Pi_\epsilon\in\mathbb{R}$ is independent of $(x,y)$.
Consequently,
\[
    p^\epsilon(x+2\pi,y)-p^\epsilon(x,y)
    =
    2\pi\Pi_\epsilon,
    \qquad
    \nabla p^\epsilon(x+2\pi,y)
    =
    \nabla p^\epsilon(x,y).
\]
The coefficient $\Pi_\epsilon$ represents the mean horizontal
pressure gradient and is determined as part of the construction.
\end{remark}

Now, we briefly introduce the strategy and some key ideas for proving Theorem \ref{total theorem}. \\
\indent \textbf{Step 1: Construction of an approximate solution.} We construct an approximate solution $ (u^a, v^a, p^a) $ by the matched asymptotic expansions, which contain the Euler part $ (u_e^a, v_e^a, p_e^a) $, the upper boundary layer part $ (u_p^a, v_p^a, p_p^a) $ and the lower boundary layer part $ (\hat{u}_p^a, \hat{v}_p^a, \hat{p}_p^a) $. Since the leading Euler profile $u_e(y)$ vanishes linearly at
both walls, the boundary layers have thickness of order
$\epsilon^{1/3}$, and the expansion is organized in powers of
$\epsilon^{1/3}$. The upper and lower boundary layer corrections
are constructed so that the resulting approximate solution
preserves the midline symmetry.

The Euler part satisfies, for all nonnegative integers $j$ and $k$, $ j, k\in \mathbb{N} $, 
\begin{align} \label{estimate of Euler}
    \left \| \partial_x^j\partial_y^k(u_e^a-u_e(y)) \right \|_{L^{\infty}}+ \left \| \partial_x^j\partial_y^kv_e^a \right \|_{L^{\infty}}\le C_{j,k}\epsilon.
\end{align}
For the upper boundary layer part, we have, for any $ j, k, l \in \mathbb{N} $,
\begin{align} \label{estimate of upper Prandtl}
    \left \| \xi^l\partial_x^j\partial_{\xi}^ku_p^a \right \|_{L^{\infty}}\le C_{j,k,l}\epsilon,\quad \left \| \xi^l\partial_x^j\partial_{\xi}^kv_p^a \right \|_{L^{\infty}}\le C_{j,k,l}\epsilon^{\frac{4}{3}}
\end{align}
Similarly, for the lower boundary layer part, we have, for any $ j, k, l \in \mathbb{N} $,
\begin{align} \label{estimate of lower Prandtl}
    \left \| \gamma^l\partial_x^j\partial_{\gamma}^k\hat{u}_p^a \right \|_{L^{\infty}}\le C_{j,k,l}\epsilon,\quad \left \| \gamma^l\partial_x^j\partial_{\gamma}^k\hat{v}_p^a \right \|_{L^{\infty}}\le C_{j,k,l}\epsilon^{\frac{4}{3}}.
\end{align}
Here, $\xi$ and $\gamma$ denote the stretched normal variables
near the upper and lower walls, respectively, and all constants
are independent of $\epsilon$.
The construction of the approximate solution and the proofs of
these estimates are presented in Section\ref{sec2}.

\indent \textbf{Step 2: Linear stability estimate for the error equations.} We denote the error function by $ (u, v, p):=(u^{\epsilon}-u^a, v^{\epsilon}-v^a, p^{\epsilon}-p^a) $ and there holds
\begin{align} \label{error equation}
    \left \{
    \begin{aligned}
        &-\epsilon\triangle u+\partial_xp+S_u=R_u, \\
        &-\epsilon\triangle v+\partial_yp+S_v=R_v, \\
        &\partial_xu+\partial_yv=0, \\
        &u(x, y)=u(x+2\pi, y), v(x, y)= v(x+2\pi, y), \\
        &u(x, 1)=0, v(x, 1)=0, \\
        &u(x, 0)=0, v(x, 0)=0,
    \end{aligned}
    \right .
\end{align}
where 
\[
\begin{aligned}
S_u := u^au_x+v^au_y+uu_x^a+vu_y^a, ~~
S_v := u^av_x+v^av_y+uv_x^a+vv_y^a
\end{aligned}
\]
are linearized convective terms. In the linear analysis, $R_u$ and $R_v$ are treated as prescribed
source terms. In the full nonlinear error system, they contain
both the residual of the approximate solution and the quadratic
terms in the error.

Our linear stability analysis combines a positivity estimate
with two basic energy estimates. The basic energy estimates
alone do not close, because the viscous dissipation is multiplied
by the small parameter $\epsilon$. The key observation is that
the leading transport terms $u^a u_x$ and $u^a v_x$ yield
positive control of the horizontal derivatives without a
prefactor of $\epsilon$.

More precisely, the estimates for the approximate solution,
together with the matching and wall conditions, imply
\[
    \bigl(1-C\epsilon^{2/3}\bigr)u_e(y)
    \leq u^a(x,y)
    \leq (1+C\epsilon^{2/3}) u_e(y)
\]
for sufficiently small $\epsilon$.
This comparison allows us to derive an a priori estimate of
the form
\begin{align}\label{aprior estimate on error}
\begin{aligned}
    \int_\Omega
    \Big[
        &u_e(y)\bigl(u_x^2+v_x^2\bigr)
        +\epsilon
        \bigl(u_{0,y}^2+u_e(y)u_y^2\bigr)+\epsilon^2
        \bigl(
            u_{xx}^2+u_{xy}^2
            +v_{xx}^2+v_{xy}^2
        \bigr)
    \Big]\,dx\,dy
    \lesssim \mathcal Q_\epsilon,
\end{aligned}
\end{align}
where
\[
    u_0(y)
    :=
    \frac{1}{2\pi}\int_0^{2\pi}u(x,y)\,dx,
\]
and $\mathcal Q_\epsilon$ denotes the source contributions
arising in the combined estimates. The precise estimate is
established in Proposition~\ref{pro-linear stability estimate}.

We briefly explain the mechanism behind the positivity estimate. Multiplying the first and second momentum equations in
\eqref{error equation} by $u_x$ and $v_x$, respectively,
adding the resulting identities, and integrating over $\Omega$,
we obtain, using incompressibility, periodicity, and integration
by parts,
\begin{align}\label{positive0}
\begin{aligned}
    \int_\Omega u^a\bigl(u_x^2+v_x^2\bigr)\,dx\,dy
    \leq
    \left|
        \int_\Omega
        u_y^a v u_x
             \,dx\,dy
    \right|
    +\mathcal E_1.
\end{aligned}
\end{align}
Here, $\mathcal E_1$ collects the source contributions and the
remaining linear terms, which are controlled in the subsequent
estimates. A difficulty is that the left-hand side of
\eqref{positive0} does not directly control $u_y$.
Consequently, the terms involving $u$ and $u_y$ must be
estimated carefully.

The pointwise comparison between $u^a$ and $u_e$ gives
\[
    \int_\Omega u^a\bigl(u_x^2+v_x^2\bigr)\,dx\,dy
    \geq
    \bigl(1-C\epsilon^{2/3}\bigr)
    \int_\Omega u_e(y)\bigl(u_x^2+v_x^2\bigr)\,dx\,dy.
\]
On the other hand, using
\eqref{estimate of Euler},
\eqref{estimate of upper Prandtl}, 
\eqref{estimate of lower Prandtl}, and the incompressibility condition,  we obtain
\begin{align}
\Big|\int_{\Omega}u_y^avu_x\mathrm{d}x\mathrm{d}y\Big|=&\frac{1}{2}\Big|\int_{\Omega}u_{yy}^av^2\mathrm{d}x\mathrm{d}y\Big|\nonumber\\
\leq& \frac{1}{2}\int_{\Omega}|(u^a-u_e(y))_{yy}v^2\mathrm{d}x\mathrm{d}y+\frac{1}{2}\int_{\Omega}|u_e''(y)|v^2\mathrm{d}x\mathrm{d}y\nonumber\\
\leq& C\epsilon^{\frac13}\int_{\Omega} v^2\mathrm{d}x\mathrm{d}y+\int_{\Omega}\alpha v^2\mathrm{d}x\mathrm{d}y.\nonumber
\end{align}
For the Couette profile $u_e(y)=\alpha y$ in \cite{FPZ}, there holds $u_e''(y)=0$,
so the curvature term vanishes. In contrast, the Poiseuille
profile $u_e(y)=\alpha y(1-y)$ has nonzero curvature,
$u_e''(y)=-2\alpha$, which gives rise to the additional contribution
\[
    \frac12\int_\Omega |u_e''(y)|v^2\,dx\,dy
    =
    \alpha\int_\Omega v^2\,dx\,dy.
\]
The odd symmetry of $v$ with respect to the midline $y=\frac12$
is essential for controlling this term through a weighted
Poincar\'e inequality, thereby allowing it to be absorbed into
the positive part of the estimate.
Indeed, since
\[
    v(x,y)=-v(x,1-y),
\]
the following weighted Poincar\'e inequality holds:
\[
    \alpha\int_\Omega v^2\,dx\,dy
    \leq
    \ln 2\int_\Omega u_e(y)v_y^2\,dx\,dy.
\]
Using the incompressibility condition $v_y=-u_x$, we deduce
\[
    \alpha\int_\Omega v^2\,dx\,dy
    \leq
    \ln 2\int_\Omega u_e(y)u_x^2\,dx\,dy.
\]
Since
\[
    \frac{3}{100}+\ln 2<1,
\]
this contribution can be absorbed into the left-hand side of
the positivity estimate for sufficiently small $\epsilon$.
The additional term
$C\epsilon^{1/3}\int_\Omega v^2\,dx\,dy$
is absorbed in the same way, while the term involving
$u_{0,y}$ is controlled by the basic energy estimates. Combining these estimates yields
\eqref{aprior estimate on error}.
The detailed argument is given in
Proposition~\ref{pro-linear stability estimate}.\\
\indent \textbf{Step 3: Existence of a solution to the error equations.} After establishing solvability of the linearized problem,
we construct a solution to the full nonlinear error system
by a contraction mapping argument.

To distinguish the approximation residual from the nonlinear
error terms, write
\[
    \boldsymbol U^a=(u^a,v^a),
    \qquad
    \boldsymbol F=(F_u,F_v),
\]
and define
\[
    \boldsymbol R^a
    :=
    (\boldsymbol U^a\cdot\nabla)\boldsymbol U^a
    +\nabla p^a
    -\epsilon\Delta\boldsymbol U^a
    -\epsilon\boldsymbol F.
\]
Then the right-hand side of the nonlinear error system is
\[
    (R_u,R_v)
    =
    -\boldsymbol R^a
    -
    \bigl(
        u u_x+v u_y,\,
        u v_x+v v_y
    \bigr).
\]

We define an iteration map by evaluating the quadratic terms
at a given error field and solving the resulting linearized
system. The a priori estimate
\eqref{aprior estimate on error}, together with the bounds
on the approximation residual and the nonlinear terms,
shows that this map sends a suitably chosen closed ball
in the midline-symmetric function space into itself and is
a contraction for sufficiently small $\epsilon$.
Its fixed point provides a symmetric solution to the nonlinear
error equations. Adding this error to the approximate solution
then yields the desired steady Navier--Stokes solution. \\

\indent Our paper is organized as follows. In section \ref{sec2}, we construct an approximate solution by the matched asymptotic expansions. In section \ref{sec3}, based on the approximate solution, we obtain the error equations, establish the linear stability estimate which consists of the positive estimate and basic energy estimates. In section \ref{sec4}, we prove the existence of a solution to the error equations by employing the Sobolev embedding and the contraction mapping principle.

\section{Construction of approximate solutions}\label{sec2}
In this section, we construct an approximate solution of the Navier-Stokes equations (\ref{NS equation}) with boundary condition (\ref{D-BC}) by the matched asymptotic expansions. 
\subsection{Euler expansions} Away from the boundary, we make the following formal expansions
\begin{align} \label{Euelr expansion}
\begin{aligned}
    u^{\epsilon }(x, y) &=u_e(y)+\sum_{k=1}^{k_0} \epsilon ^{\frac{k}{3}}u_{e}^{(\frac{k}{3})}(x, y) + h.o.t, \\
    v^{\epsilon }(x, y) &=\sum_{k=1}^{k_0} \epsilon ^{\frac{k}{3}}v_{e}^{(\frac{k}{3})}(x, y) + h.o.t, \\
    p^{\epsilon }(x, y) &=\sum_{k=1}^{k_0} \epsilon ^{\frac{k}{3}}p_{e}^{(\frac{k}{3})}(x, y) + h.o.t.
\end{aligned}
\end{align}
 Here and in what follows, ``h.o.t" means higher order terms.
\subsubsection{Equation for $ (u_e^{(\frac{1}{3})}, v_e^{(\frac{1}{3})}, p_e^{(\frac{1}{3})}) $} By substituting the Euler expansions (\ref{Euelr expansion}) into (\ref{NS equation}) and collecting the $\epsilon^{\frac{1}{3}}$ order terms, we have the following linearized steady Euler equations for $ (u_e^{(\frac{1}{3})}, v_e^{(\frac{1}{3})}, p_e^{(\frac{1}{3})}) $
\begin{align} \label{eque1}
\left \{
\begin{aligned}
      &u_e(y)\partial_xu_e^{(\frac{1}{3})}+u'_e(y)v_e^{(\frac{1}{3})}+\partial_xp_e^{(\frac{1}{3})}=0, \\
    &u_e(y)\partial_xv_e^{(\frac{1}{3})}+\partial_yp_e^{(\frac{1}{3})}=0, \\
    &\partial_xu_e^{(\frac{1}{3})}+\partial_yv_e^{(\frac{1}{3})}=0, \\  
    &(u_e^{(\frac{1}{3})}, v_e^{(\frac{1}{3})})(x, y)=(u_e^{(\frac{1}{3})}, v_e^{(\frac{1}{3})})(x+2\pi, y), \\
    &v_e^{(\frac{1}{3})}|_{y=1}=-v_p^{(\frac{1}{3})}|_{\xi =0},~~ v_e^{(\frac{1}{3})}|_{y=0}=-\hat{v}_p^{(\frac{1}{3})}|_{\gamma  =0},
\end{aligned}
\right .\
\end{align}
where $(v_p^{(\frac{1}{3})}, \hat{v}_p^{(\frac{1}{3})})$ is defined in (\ref{equp0}) and (\ref{eqhup0}) respectively in next subsection. In fact, we will take $(u_e^{(\frac{1}{3})}, v_e^{(\frac{1}{3})}, p_e^{(\frac{1}{3})})=0$.
\subsubsection{Equation for $ (u_e^{(\frac{2}{3})}, v_e^{(\frac{2}{3})}, p_e^{(\frac{2}{3})}) $} By substituting the Euler expansions (\ref{Euelr expansion}) into (\ref{NS equation}) and collecting the $\epsilon^{\frac{2}{3}}$ order terms, we have the following linearized steady Euler equations for $ (u_e^{(\frac{2}{3})}, v_e^{(\frac{2}{3})}, p_e^{(\frac{2}{3})}) $
\begin{align} \label{eque2}
\left \{
\begin{aligned}
      &u_e(y)\partial_xu_e^{(\frac{2}{3})}+u'_e(y)v_e^{(\frac{2}{3})}+\partial_xp_e^{(\frac{2}{3})}=0, \\
    &u_e(y)\partial_xv_e^{(\frac{2}{3})}+\partial_yp_e^{(\frac{2}{3})}=0, \\
    &\partial_xu_e^{(\frac{2}{3})}+\partial_yv_e^{(\frac{2}{3})}=0, \\  
    &(u_e^{(\frac{2}{3})}, v_e^{(\frac{2}{3})})(x, y)=(u_e^{(\frac{2}{3})}, v_e^{(\frac{2}{3})})(x+2\pi, y), \\
    &v_e^{(\frac{2}{3})}|_{y=1}=-v_p^{(\frac{2}{3})}|_{\xi =0},~~ v_e^{(\frac{2}{3})}|_{y=0}=-\hat{v}_p^{(\frac{2}{3})}|_{\gamma  =0},
\end{aligned}
\right .\
\end{align}
where $(v_p^{(\frac{2}{3})}, \hat{v}_p^{(\frac{2}{3})})$ is defined in (\ref{equp1}) and (\ref{eqhup1}) respectively in next subsection. In fact, we will take $(u_e^{(\frac{2}{3})}, v_e^{(\frac{2}{3})}, p_e^{(\frac{2}{3})})=0$.
\subsubsection{Equation for $ (u_e^{(1)}, v_e^{(1)}, p_e^{(1)}) $} By substituting the Euler expansions (\ref{Euelr expansion}) into (\ref{NS equation}) and collecting the $\epsilon^{1}$ order terms, we have the following linearized steady Euler equations for $ (u_e^{(1)}, v_e^{(1)}, p_e^{(1)}) $
\begin{align} \label{eque3}
\left \{
\begin{aligned}
      &u_e(y)\partial_xu_e^{(1)}+u'_e(y)v_e^{(1)}+\partial_xp_e^{(1)}-u_e''(y)=F_u, \\
    &u_e(y)\partial_xv_e^{(1)}+\partial_yp_e^{(1)}=F_v, \\
    &\partial_xu_e^{(1)}+\partial_yv_e^{(1)}=0, \\  
    &(u_e^{(1)}, v_e^{(1)})(x, y)=(u_e^{(1)}, v_e^{(1)})(x+2\pi, y), \\
    &v_e^{(1)}|_{y=1}=-v_p^{(1)}|_{\xi =0},~~ v_e^{(1)}|_{y=0}=-\hat{v}_p^{(1)}|_{\gamma  =0},
\end{aligned}
\right .\
\end{align}
where $(v_p^{(1)}, \hat{v}_p^{(1)})$ is defined in (\ref{equp2}) and (\ref{eqhup2}) respectively in next subsection. \\

\subsection{Prandtl expansions near y = 1} Define the upper boundary layer variable $ \xi = \frac{y-1}{\epsilon^{\frac{1}{3}}} $. We make the following Prandtl expansions near $ y = 1$

\begin{align} \label{upper boundary layer expansion}
\begin{aligned}
    u^{\epsilon }(x, y) &=u_e(y)+u_p^{(0)}(x,\xi)+\sum_{k=1}^{k_0} \epsilon ^{\frac{k}{3}}\big(u_{e}^{(\frac{k}{3})}(x, y) + u_p^{(\frac{k}{3})}(x, \xi) \big) + h.o.t, \\
    v^{\epsilon }(x, y) &=\sum_{k=1}^{k_0} \epsilon ^{\frac{k}{3}}\big(v_{e}^{(\frac{k}{3})}(x, y)+v_p^{(\frac{k}{3})}(x, \xi) \big) + h.o.t, \\
    p^{\epsilon }(x, y) &=\sum_{k=1}^{k_0} \epsilon ^{\frac{k}{3}}\big(p_{e}^{(\frac{k}{3})}(x, y)+p_p^{(\frac{k}{3})}(x, \xi)\big) + h.o.t.
\end{aligned}
\end{align}
where as $\xi\rightarrow -\infty$
\begin{align}
  \partial_\theta^l\partial_\xi^mv_p^{(i)}(\theta,\xi)\rightarrow 0, \ \partial_\theta^l\partial_\xi^mp_p^{(i)}(\theta,\xi)\rightarrow 0,\label{matching condition}
\end{align}
here $l,m\geq 0, i=0,1,\cdots,$ and the boundary condition is matched by
\begin{align} \label{A-BC1}
    \begin{aligned}
        u_e^{(\frac{k}{3})}|_{y=1}+u_p^{(\frac{k}{3})}|_{\xi=0} = 0,~~ 
        v_e^{(\frac{k}{3})}|_{y=1}+v_p^{(\frac{k}{3})}|_{\xi=0} = 0,~~ 0 \le k \in \mathbb{N}. 
    \end{aligned}
\end{align}
\subsubsection{Equation for $ (u_p^{(0)}, v_p^{(\frac{1}{3})},p_p^{(\frac{1}{3})}) $} By substituting the upper boundary layer expansion (\ref{upper boundary layer expansion}) into (\ref{NS equation}) and collecting the $ \epsilon^0 $ order terms, we have the following steady boundary layer equations for $ (u_p^{(0)}, v_p^{(\frac{1}{3})},p_p^{(\frac{1}{3})}) $

\begin{align} \label{equp0}
\left \{
    \begin{aligned}
        &u_p^{(0)}\partial_xu_p^{(0)}+(v_p^{(\frac{1}{3})}-v_p^{(\frac{1}{3})}(x, 0))\partial_{\xi}u_p^{(0)} =0, \\
        &\partial_x u_p^{(0)}+ \partial_{\xi}v_p^{(\frac{1}{3})}=0, \\
        &\partial_{\xi}p_p^{(\frac{1}{3})}=0, \\
        &(u_p^{(0)}, v_p^{(\frac{1}{3})})(x, \xi) = (u_p^{(0)}, v_p^{(\frac{1}{3})})(x+2\pi, \xi), \\
        &u_p^{(0)}|_{\xi = 0} = 0, \\
        &\lim\limits_{\xi \to -\infty} (v_p^{(\frac{1}{3})}, p_p^{(\frac{1}{3})}) = (0,0).
    \end{aligned}
\right .
\end{align}
Here we take $ (u_p^{(0)}, v_p^{(\frac{1}{3})},p_p^{(\frac{1}{3})}) = (0,0,0)$.

\subsubsection{Equation for $ (u_p^{(\frac{1}{3})}, v_p^{(\frac{2}{3})},p_p^{(\frac{2}{3})}) $} By substituting the upper boundary layer expansion (\ref{upper boundary layer expansion}) into (\ref{NS equation}) and collecting the $ \epsilon^{\frac{1}{3}}/\epsilon^{\frac{2}{3}} $ order terms, we have the following steady boundary layer equations for $ (u_p^{(\frac{1}{3})}, v_p^{(\frac{2}{3})}, p_p^{(\frac{2}{3})}) $
\begin{align}\label{equp1}
    \left \{
    \begin{aligned}
        &(-\alpha\xi+u_p^{(\frac{1}{3})})\partial_xu_p^{(\frac{1}{3})}-\alpha v_p^{(\frac{2}{3})}+(v_p^{(\frac{2}{3})}-v_p^{(\frac{2}{3})}(x, 0))\partial_{\xi}u_p^{(\frac{1}{3})}+\partial_xp_p^{(\frac{2}{3})}-\partial_{\xi \xi}u_p^{(\frac{1}{3})}=0, \\
        &\partial_x u_p^{(\frac{1}{3})}+ \partial_{\xi}v_p^{(\frac{2}{3})}=0, \\
        &\partial_{\xi}p_p^{(\frac{2}{3})}=0, \\
        &(u_p^{(\frac{1}{3})}, v_p^{(\frac{2}{3})})(x, \xi) = (u_p^{(\frac{1}{3})}, v_p^{(\frac{2}{3})})(x+2\pi, \xi), \\
        &u_p^{(\frac{1}{3})}|_{\xi = 0} = 0, \\
        &\lim\limits_{\xi \to -\infty} (v_p^{(\frac{2}{3})}, p_p^{(\frac{2}{3})}) = (0,0).
    \end{aligned}
    \right .
\end{align}
Similarly, we take $ (u_p^{(\frac{1}{3})}, v_p^{(\frac{2}{3})},p_p^{(\frac{2}{3})}) =(0,0,0) $.

\subsubsection{Equation for $ (u_p^{(\frac{2}{3})}, v_p^{(1)},p_p^{(1)}) $} By substituting the upper boundary layer expansion (\ref{upper boundary layer expansion}) into (\ref{NS equation}) and collecting the $ \epsilon^{\frac{2}{3}}/\epsilon^1 $ order terms, we have the following steady boundary layer equations for $ (u_p^{(\frac{2}{3})}, v_p^{(1)}, p_p^{(1)}) $
\begin{align}\label{equp2}
    \left \{
    \begin{aligned}
        &-\alpha \xi \partial_x u_p^{(\frac{2}{3})}-\alpha v_p^{(1)}-\partial_{\xi \xi} u_p^{(\frac{2}{3})}+\partial_xp_p^{(1)}=0, \\
        &\partial_x u_p^{(\frac{2}{3})}+ \partial_{\xi}v_p^{(1)}=0, \\
        &\partial_{\xi}p_p^{(1)}=0, \\
        &(u_p^{(\frac{2}{3})}, v_p^{(1)})(x, \xi) = (u_p^{(\frac{2}{3})}, v_p^{(1)})(x+2\pi, \xi), \\
        &u_p^{(\frac{2}{3})}|_{\xi = 0} = 0, \\
        &\lim\limits_{\xi \to -\infty} (v_p^{(1)}, p_p^{(1)}) = (0,0).
    \end{aligned}
    \right .
\end{align}
Similarly, we take $ (u_p^{(\frac{2}{3})}, v_p^{(1)},p_p^{(1)}) =(0,0,0)$.
\subsubsection{Equation for $ (u_p^{(1)}, v_p^{(\frac{4}{3})},p_p^{(\frac{4}{3})}) $} By substituting the upper boundary layer expansion (\ref{upper boundary layer expansion}) into (\ref{NS equation}) and collecting the $ \epsilon^{1}/\epsilon^{\frac{4}{3}} $ order terms, we have the following steady boundary layer equations for $ (u_p^{(1)}, v_p^{(\frac{4}{3})}, p_p^{(\frac{4}{3})}) $
\begin{align}\label{equp3}
    \left \{
    \begin{aligned}
        &-\alpha \xi \partial_x u_p^{(1)}-\alpha v_p^{(\frac{4}{3})}+\partial_xp_p^{(\frac{4}{3})}-\partial_{\xi \xi} u_p^{(1)}=0, \\
        &\partial_x u_p^{(1)}+ \partial_{\xi}v_p^{(\frac{4}{3})}=0, \\
        &\partial_{\xi}p_p^{(\frac{4}{3})}=0, \\
        &(u_p^{(1)}, v_p^{(\frac{4}{3})})(x, \xi) = (u_p^{(1)}, v_p^{(\frac{4}{3})})(x+2\pi, \xi), \\
        &u_p^{(1)}|_{\xi = 0} = -u_e^{(1)}|_{y=1}, \\
        &\lim\limits_{\xi \to -\infty} (v_p^{(\frac{4}{3})}, p_p^{(\frac{4}{3})}) =(0,0).
    \end{aligned}
    \right .
\end{align}

\subsection{Prandtl expansions near y = 0} Similarly as above, we define the lower boundary layer variable $\gamma = \frac{y}{\epsilon^{\frac{1}{3}}}$ and make the following Prandtl expansions near y = 0
\begin{align} \label{lower boundary layer expansion}
\begin{aligned}
    u^{\epsilon }(x, y) &=u_e(y)+\hat{u}_p^{(0)}(x,\gamma)+\sum_{k=1}^{k_0} \epsilon ^{\frac{k}{3}}\big(u_{e}^{(\frac{k}{3})}(x, y) + \hat{u}_p^{(\frac{k}{3})}(x, \gamma) \big) + h.o.t, \\
    v^{\epsilon }(x, y) &=\sum_{k=1}^{k_0} \epsilon ^{\frac{k}{3}}\big(v_{e}^{(\frac{k}{3})}(x, y)+\hat{v}_p^{(\frac{k}{3})}(x, \gamma) \big) + h.o.t, \\
    p^{\epsilon }(x, y) &=\sum_{k=1}^{k_0} \epsilon ^{\frac{k}{3}}\big(p_{e}^{(\frac{k}{3})}(x, y)+\hat{p}_p^{(\frac{k}{3})}(x, \gamma)\big) + h.o.t.
\end{aligned}
\end{align}
where as $\gamma\rightarrow +\infty$
\begin{align}
  \partial_\theta^l\partial_\gamma^m\hat{v}_p^{(i)}(\theta,\gamma)\rightarrow 0, \ \partial_\theta^l\partial_\gamma^m\hat{p}_p^{(i)}(\theta,\gamma)\rightarrow 0,\label{matching condition}
\end{align}
here $l,m\geq 0, i=0,1,\cdots,$ and the boundary condition is matched by
\begin{align} \label{A-BC2}
    \begin{aligned}
        u_e^{(\frac{k}{3})}|_{y=0}+\hat{u}_p^{(\frac{k}{3})}|_{\gamma=0} = 0,~~v_e^{(\frac{k}{3})}|_{y=0}+\hat{v}_p^{(\frac{k}{3})}|_{\gamma=0} = 0,~~ 0 \le k \in \mathbb{N}. 
    \end{aligned}
\end{align}

\subsubsection{Equation for $ (\hat{u}_p^{(0)}, \hat{v}_p^{(\frac{1}{3})},\hat{p}_p^{(\frac{1}{3})}) $} By substituting the lower boundary layer expansion (\ref{lower boundary layer expansion}) into (\ref{NS equation}) and collecting the $ \epsilon^0 $ order terms, we have the following steady boundary layer equations for $ (\hat{u}_p^{(0)}, \hat{v}_p^{(\frac{1}{3})},\hat{p}_p^{(\frac{1}{3})}) $

\begin{align} \label{eqhup0}
\left \{
    \begin{aligned}
        &\hat{u}_p^{(0)}\partial_x\hat{u}_p^{(0)}+(\hat{v}_p^{(\frac{1}{3})}-\hat{v}_p^{(\frac{1}{3})}(x, 0))\partial_{\gamma}\hat{u}_p^{(0)} =0, \\
        &\partial_x \hat{u}_p^{(0)}+ \partial_{\gamma}\hat{v}_p^{(\frac{1}{3})}=0, \\
        &\partial_{\gamma}\hat{p}_p^{(\frac{1}{3})}=0, \\
        &(\hat{u}_p^{(0)}, \hat{v}_p^{(\frac{1}{3})})(x, \gamma) = (\hat{u}_p^{(0)}, \hat{v}_p^{(\frac{1}{3})})(x+2\pi, \gamma), \\
        &\hat{u}_p^{(0)}|_{\gamma = 0} = 0, \\
        &\lim\limits_{\gamma \to +\infty} (\hat{v}_p^{(\frac{1}{3})}, \hat{p}_p^{(\frac{1}{3})}) = (0,0).
    \end{aligned}
\right .
\end{align}
Here we take $ (\hat{u}_p^{(0)}, \hat{v}_p^{(\frac{1}{3})},\hat{p}_p^{(\frac{1}{3})}) =(0,0,0)$.

\subsubsection{Equation for $ (\hat{u}_p^{(\frac{1}{3})}, \hat{v}_p^{(\frac{2}{3})},\hat{p}_p^{(\frac{2}{3})}) $} By substituting the lower boundary layer expansion (\ref{lower boundary layer expansion}) into (\ref{NS equation}) and collecting the $ \epsilon^{\frac{1}{3}}/\epsilon^{\frac{2}{3}} $ order terms, we have the following steady boundary layer equations for $ (\hat{u}_p^{(\frac{1}{3})}, \hat{v}_p^{(\frac{2}{3})}, \hat{p}_p^{(\frac{2}{3})}) $
\begin{align}\label{eqhup1}
    \left \{
    \begin{aligned}
        &(\alpha\gamma+\hat{u}_p^{(\frac{1}{3})})\partial_x\hat{u}_p^{(\frac{1}{3})}+\alpha \hat{v}_p^{(\frac{2}{3})}+(\hat{v}_p^{(\frac{2}{3})}-\hat{v}_p^{(\frac{2}{3})}(x,0))\partial_{\gamma}\hat{u}_p^{(\frac{1}{3})}+\partial_xp_p^{(\frac{2}{3})}-\partial_{\gamma \gamma}\hat{u}_p^{(\frac{1}{3})}=0, \\
        &\partial_x \hat{u}_p^{(\frac{1}{3})}+ \partial_{\gamma}\hat{v}_p^{(\frac{2}{3})}=0, \\
        &\partial_{\gamma}\hat{p}_p^{(\frac{2}{3})}=0, \\
        &(\hat{u}_p^{(\frac{1}{3})}, \hat{v}_p^{(\frac{2}{3})})(x, \gamma) = (\hat{u}_p^{(\frac{1}{3})}, \hat{v}_p^{(\frac{2}{3})})(x+2\pi, \gamma), \\
        &\hat{u}_p^{(\frac{1}{3})}|_{\gamma = 0} = 0, \\
        &\lim\limits_{\gamma \to +\infty} ( \hat{v}_p^{(\frac{2}{3})}, \hat{p}_p^{(\frac{2}{3})}) = (0,0).
    \end{aligned}
    \right .
\end{align}
Similarly, we take $ (\hat{u}_p^{(\frac{1}{3})}, \hat{v}_p^{(\frac{2}{3})},\hat{p}_p^{(\frac{2}{3})}) =(0,0,0) $.

\subsubsection{Equation for $ (\hat{u}_p^{(\frac{2}{3})}, \hat{v}_p^{(1)},\hat{p}_p^{(1)}) $} By substituting the lower boundary layer expansion (\ref{lower boundary layer expansion}) into (\ref{NS equation}) and collecting the $ \epsilon^{\frac{2}{3}}/\epsilon^1 $ order terms, we have the following steady boundary layer equations for $ (\hat{u}_p^{(\frac{2}{3})}, \hat{v}_p^{(1)}, \hat{p}_p^{(1)}) $
\begin{align}\label{eqhup2}
    \left \{
    \begin{aligned}
        &\alpha \gamma \partial_x \hat{u}_p^{(\frac{2}{3})}+\alpha \hat{v}_p^{(1)}-\partial_{\gamma \gamma} \hat{u}_p^{(\frac{2}{3})}+\partial_x\hat{p}_p^{(1)}=0, \\
        &\partial_x \hat{u}_p^{(\frac{2}{3})}+ \partial_{\gamma}\hat{v}_p^{(1)}=0, \\
        &\partial_{\gamma}\hat{p}_p^{(1)}=0, \\
        &(\hat{u}_p^{(\frac{2}{3})}, \hat{v}_p^{(1)})(x, \gamma) = (\hat{u}_p^{(\frac{2}{3})}, \hat{v}_p^{(1)})(x+2\pi, \gamma), \\
        &\hat{u}_p^{(\frac{2}{3})}|_{\gamma = 0} = 0, \\
        &\lim\limits_{\gamma \to +\infty} (\hat{v}_p^{(1)}, \hat{p}_p^{(1)}) = (0,0).
    \end{aligned}
    \right .
\end{align}
Similarly, we take $ (\hat{u}_p^{(\frac{2}{3})}, \hat{v}_p^{(1)},\hat{p}_p^{(1)}) =(0,0,0)$.
\subsubsection{Equation for $ (\hat{u}_p^{(1)}, \hat{v}_p^{(\frac{4}{3})},\hat{p}_p^{(\frac{4}{3})}) $} By substituting the lower boundary layer expansion (\ref{lower boundary layer expansion}) into (\ref{NS equation}) and collecting the $ \epsilon^{1}/\epsilon^{\frac{4}{3}} $ order terms, we have the following steady boundary layer equations for $ (\hat{u}_p^{(1)}, \hat{v}_p^{(\frac{4}{3})}, \hat{p}_p^{(\frac{4}{3})}) $
\begin{align}\label{eqhup3}
    \left \{
    \begin{aligned}
        &\alpha \gamma \partial_x \hat{u}_p^{(1)}+\alpha \hat{v}_p^{(\frac{4}{3})}+\partial_x\hat{p}_p^{(\frac{4}{3})}-\partial_{\gamma \gamma} \hat{u}_p^{(1)}=0, \\
        &\partial_x \hat{u}_p^{(1)}+ \partial_{\gamma}\hat{v}_p^{(\frac{4}{3})}=0, \\
        &\partial_{\gamma}\hat{p}_p^{(\frac{4}{3})}=0, \\
        &(\hat{u}_p^{(1)}, \hat{v}_p^{(\frac{4}{3})})(x, \gamma) = (\hat{u}_p^{(1)}, \hat{v}_p^{(\frac{4}{3})})(x+2\pi, \gamma), \\
        &\hat{u}_p^{(1)}|_{\gamma = 0} = -u_e^{(1)}|_{y=0}, \\
        &\lim\limits_{\gamma \to +\infty} ( \hat{v}_p^{(\frac{4}{3})}, \hat{p}_p^{(\frac{4}{3})}) = (0,0).
    \end{aligned}
    \right .
\end{align}

\begin{remark}
Let us now explain the reason why we choose the thickness of the boundary layer to be $\epsilon^{\frac{1}{3}}$. In fact, we let
\[
\begin{aligned}
&u^{\epsilon}=u_e(y)+\epsilon^{\beta}(u_e^{(\beta)}+u_p^{(\beta)}), \\
&v^{\epsilon}=\epsilon^{\beta}(v_e^{(\beta)}+v_p^{(\beta)}).
\end{aligned}
\]
Then we take it into (\ref{NS equation}), we can see
\[
-\epsilon^{2\beta}\alpha \xi \partial_xu_p^{(\beta)}-\epsilon^{1-\beta}\partial_{\xi \xi}u_p^{(\beta)}=0,
\]
so to match the order, we need $2\beta = 1-\beta$, i.e. $\beta = \frac{1}{3}$. 
\end{remark}

\subsection{Solvability of Euler equations and boundary layer equations} The order in which we solve the equation is as follows

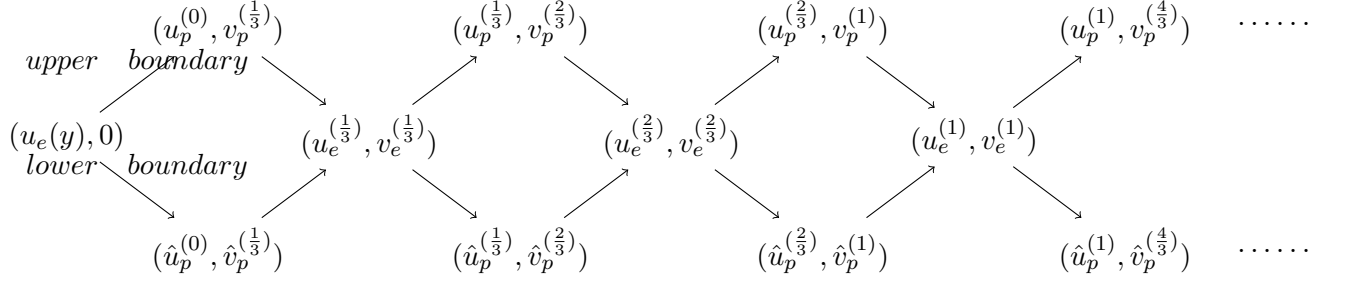
\begin{tikzpicture}
    \node (X) at (2,0) {$(u_e(y),0)$};
    \node (A) at (4,1.5) {$(u_p^{(0)}, v_p^{(\frac{1}{3})})$};
    \node (B) at (4,-1.5) {$(\hat{u}_p^{(0)}, \hat{v}_p^{(\frac{1}{3})})$};
    \node (C) at (6,0) {$(u_e^{(\frac{1}{3})},v_e^{(\frac{1}{3})})$};
    \node (D) at (8,1.5) {$(u_p^{(\frac{1}{3})}, v_p^{(\frac{2}{3})})$};
    \node (E) at (8,-1.5) {$(\hat{u}_p^{(\frac{1}{3})}, \hat{v}_p^{(\frac{2}{3})})$};
    \node (F) at (10,0) {$(u_e^{(\frac{2}{3})},v_e^{(\frac{2}{3})})$};
    \node (G) at (12,1.5) {$(u_p^{(\frac{2}{3})}, v_p^{(1)})$};
    \node (H) at (12,-1.5) {$(\hat{u}_p^{(\frac{2}{3})}, \hat{v}_p^{(1)})$};
    \node (I) at (14,0) {$(u_e^{(1)},v_e^{(1)})$};
    \node (J) at (16,1.5) {$(u_p^{(1)}, v_p^{(\frac{4}{3})})$};
    \node (K) at (16,-1.5) {$(\hat{u}_p^{(1)}, \hat{v}_p^{(\frac{4}{3})})$};
    \node (L) at (18,1.5) {$\dots \dots$};
    \node (M) at (18,-1.5) {$\dots \dots$};
    \draw[->] (X) -- node[above] {$upper \quad boundary$} (A);
    \draw[->] (X) -- node[above] {$lower \quad boundary$} (B);
    \draw[->] (A) -- node[above] {} (C);
    \draw[->] (B) -- node[above] {} (C);
    \draw[->] (C) -- node[above] {} (D);
    \draw[->] (C) -- node[above] {} (E);
    \draw[->] (D) -- node[above] {} (F);
    \draw[->] (E) -- node[above] {} (F);
    \draw[->] (F) -- node[above] {} (G);
    \draw[->] (F) -- node[above] {} (H);
    \draw[->] (G) -- node[above] {} (I);
    \draw[->] (H) -- node[above] {} (I);
    \draw[->] (I) -- node[above] {} (J);
    \draw[->] (I) -- node[above] {} (K);

\end{tikzpicture}
We give some explanations for this picture.
\begin{itemize}
    \item We use $(u_e, v_e)$ denote the Euler asymptotic expansion away from the boundary, while $(u_p, v_p)$ to denote the asymptotic expansion near the upper boundary and $(\hat{u}_p, \hat{v}_p)$ to denote the asymptotic expansion near the lower boundary. 
    \item The superscript on the Euler and the boundary layer asymptotic expansion represents the matching order of $\epsilon$.
    \item The equations satisfied by $(u_e, v_e)$, $(u_p, v_p)$ and $(\hat{u}_p, \hat{v}_p)$ will be obtained by solving (\ref{NS equation}) by matching the order of $\epsilon$. The asymptotic expansion will be solved column by column starting from $(u_e(y), 0):=(\alpha y(1-y), 0)$.
\end{itemize}

\subsubsection{The solvability of the Euler equations for $(u_e^{(1)}, v_e^{(1)})$.}

Fisrtly, we present two types of the Hardy inequalities which will be used frequently.

\begin{lemma} 
    For function $f(y) \in C^1([0,1])$ with $f(0)=f(1)=0$, we have
    \begin{align} \label{hardy2}
        \int_{0}^{1}\frac{f^2(y)}{y^2} \mathrm{d}y \leq 4 \int_{0}^{1}(f'(y))^2 \mathrm{d}y, ~~\int_{0}^{1}\frac{f^2(y)}{(1-y)^2} \mathrm{d}y \leq 4 \int_{0}^{1}(f'(y))^2 \mathrm{d}y.
    \end{align}
\end{lemma}

\begin{lemma} 
    For function $f(y) \in C^1([0,1])$ with $f(0)=f(1)=0$ and $f(1-y)=-f(y)$, we have
    \begin{align} \label{hardy for odd function}
        \int_{0}^{1} \frac{f^2}{y(1-y)} \mathrm{d}y \leq \frac16\int_{0}^{1}(f'(y))^2 \mathrm{d}y.
    \end{align}
\end{lemma}
\begin{proof}
   Let $g(y)=y(1-y)(1-2y), y\in [0,\frac12]$, then 
   $$-g''(y)=\frac{6g}{y(1-y)}.$$
 Let $h(y):=\frac{f}{g}, y\in [0,\frac12]$,  then we deduce that
   \begin{align}
   \int_{0}^{\frac12}(f'(y))^2 \mathrm{d}y=&\int_0^{\frac12} \big(|h'(y)|^2g^2(y)+2h(y)h'(y)g(y)g'(y)+|g'(y)|^2h^2(y)\big)\mathrm{d}y\nonumber\\
   =&\int_0^{\frac12} \big(-g(y)g''(y)h^2(y)+|h'(y)|^2g^2(y)\big)\mathrm{d}y\nonumber\\
   \geq & 6\int_0^{\frac12} \frac{g^2(y)h^2(y)}{y(1-y)}\mathrm{d}y=6 \int_{0}^{\frac12} \frac{f^2}{y(1-y)} \mathrm{d}y.\nonumber
   \end{align}
   By symmetry, we deduce (\ref{hardy for odd function}).
\end{proof}

\begin{proposition} \label{existence of Euler-1}
    The linearized Euler system (\ref{eque3}) has a solution $ (u_e^{(1)}, v_e^{(1)}) $ which satisfies 
    \begin{align} \label{p283}
    \begin{aligned}
        &\|\partial_x^j \partial_y^k (u_e^{(1)}, v_e^{(1)})\|_{L^2} \leq C_{j,k},~~\forall j,k \geq 0, \\
        &\int_{0}^{2\pi} v_e^{(1)}(x,y) \mathrm{d}x = 0, ~~ \forall y \in [0,1],\\
        &u_e^{(1)}(x, y)=u_e^{(1)}(x, 1-y),~~v_e^{(1)}(x, y)= -v_e^{(1)}(x, 1-y).
    \end{aligned}
    \end{align}
\end{proposition}
\begin{proof}
Eliminate the pressure $p_e^{(1)}$, we get
    \begin{align} \label{eliminate pressure}
    \left \{
    \begin{aligned}
    &-u_e(y)\triangle v_e^{(1)}-2\alpha v_e^{(1)} = \partial_y F_u - \partial_xF_v=:\tilde{F},\\
    &\partial_xu_e^{(1)}+\partial_yv_e^{(1)}=0, \\  
    &(u_e^{(1)}, v_e^{(1)})(x, y)=(u_e^{(1)}, v_e^{(1)})(x+2\pi, y), \\
    &v_e^{(1)}|_{y=1}=0,~~ v_e^{(1)}|_{y=0}=0.
    \end{aligned}
    \right .
    \end{align}
Set $\tilde{H}(\Omega) := \left \{ f\in H_0^1(\Omega) ~|~ f(x,y)=-f(x,1-y), \forall (x, y)\in \Omega \right \} $ equipped with the standard inner product 
\[
(~f,~g~)_{\tilde{H}(\Omega)} := \int_{\Omega} fg\mathrm{d}x\mathrm{d}y+\int_{\Omega} (\partial_xf\partial_xg+\partial_yf\partial_yg)\mathrm{d}x\mathrm{d}y.
\]
 
For any $v,  w \in \tilde{H}(\Omega)$, let
\[
B(v, w):=\int_{\Omega}\Big(\nabla v\cdot \nabla w-\frac{2vw}{y(1-y)}\Big) \mathrm{d}x \mathrm{d}y.
\]
By Hardy inequality (\ref{hardy for odd function}), it's easy to obtain 
\[
|B(v,w)| \le C \|v\|_{H^1} \|w\|_{H^1}
\]
and
\begin{align}\label{coercity}
B(v,v)=&\int_{\Omega}\Big(|\nabla v|^2-\frac{2v^2}{y(1-y)}\Big)\mathrm{d}x \mathrm{d}y \nonumber\\
\geq & \|\nabla v\|^2_{L^2}-\frac{1}{3}\|\partial_yv\|^2_{L^2}\geq \frac{2}{3}\|\nabla v\|^2_{L^2}.
\end{align}
Thus, by Lax-Milgram Theorem and notice $\frac{\tilde{F}}{u_e(y)}\in \tilde{H}(\Omega)$, we deduce that there exists $ v_e^{(1)} \in \tilde{H}(\Omega)$ such that for any $w \in \tilde{H}(\Omega)$,
\begin{align} \label{weak}
B(v_e^{(1)},w)=\Big(\frac{\tilde{F}}{u_e(y)}, w\Big)_{L^2}.
\end{align}
By  (\ref{coercity}) and (\ref{weak}), we obtain
\[
\begin{aligned}
\|v_e^{(1)}\|_{H^1} \le C.
\end{aligned}
\]
 Similarly, by applying $\partial_x^j$ to $(\ref{eliminate pressure})_1$, we obtain 
$$\|\partial_x^jv_e^{(1)}\|_{L^2} \le C_j,~~\forall j\ge 1.$$ 
For $ \partial_{yy}v_e^{(1)}$, we use the equation (\ref{eliminate pressure}), 
\[
\partial_{yy}v_e^{(1)} = -\frac{\partial_yF_u-\partial_xF_v}{\alpha y(1-y)}-\frac{2v_e^{(1)}}{y(1-y)}-\partial_{xx}v_e^{(1)}
\]
So we can obtain that $\|\partial_{yy}v_e^{(1)}\|_{L^2} \le C$. For $\partial_y^kv_e^{(1)},~k\ge 3$, by using the equations (\ref{eliminate pressure}) and induction, we also get that $\|\partial_y^kv_e^{(1)}\|_{L^2}\le C_k$. Similarly, by applying $ \partial_x^j $ to $(\ref{eliminate pressure})_1$, we have $\|\partial_x^j \partial_y^kv_e^{(1)}\|_{L^2} \le C_{j,k}$. \\
\indent Then we define $u_e^{(1)}$ by
\begin{align} \label{ue1}
u_e^{(1)}(x, y)= -\int_{0}^{x} \partial_y v_e^{(1)} (\bar{x},y)\mathrm{d}\bar{x}+Z(y),
\end{align}
where $Z(y)$ satisfies 
\begin{align} \label{add condition}
\begin{aligned}
    Z'''(y)=&\frac{1}{2\pi}\Big( \int_{0}^{2\pi}\partial_y^2v_e^{(1)}(x, y)\int_{0}^{x}\partial_yv_e^{(1)}(\bar{x},y)\mathrm{d}\bar{x}\mathrm{d}x-
    \int_{0}^{2\pi}v_e^{(1)}(x,y)\int_{0}^{x}\partial_y^3v_e^{(1)}(\bar{x},y)\mathrm{d}\bar{x}\mathrm{d}x \\
    &+\int_{0}^{2\pi} \int_{0}^{x}\partial_y^4v_e^{(1)}(\bar{x},y)\mathrm{d}\bar{x}\mathrm{d}x\Big), \\
    Z(y)=&Z(1-y).
\end{aligned}
\end{align}
(\ref{ue1}) implies $u_e^{(1)}(x, y)=u_e^{(1)}(x, 1-y)$. Using the estimate of $v_e^{(1)}$, we have $\|\partial_x^j \partial_y^ku_e^{(1)}\|_{L^2}\le C_{j,k}$. \\
Lastly, we set 
\[
p_e^{(1)}=\Phi(y)+\int_{0}^{x}(F_u+u_e'')\mathrm{d}\bar{x}-u_e \int_{0}^{x}\partial_xu_e^{(1)}\mathrm{d}\bar{x}-u_e'\int_{0}^{x}v_e^{(1)}\mathrm{d}\bar{x},
\]
where $\Phi(y)$ satisfies that 
\[
\Phi'(y)+u_e \partial_xv_e^{(1)}(0, y)-F_v(0, y) = 0,~~\Phi(y)=\Phi(1-y),
\]
which implies that
\[
\left \{
\begin{aligned}
&u_e(y)\partial_xu_e^{(1)}+u'_e(y)v_e^{(1)}+\partial_xp_e^{(1)}-u_e''(y)=F_u, \\
&u_e(y)\partial_xv_e^{(1)} + \partial_yp_e^{(1)} = F_v.
\end{aligned}
\right .
\]

\begin{remark} \label{remark2.2}
    Here, we will explain why (\ref{add condition}) is required. In fact, if we consider the equations of $(u_e^{(2)},v_e^{(2)})$, i.e.
    \begin{align} \label{eque6}
    \left \{
        \begin{aligned}
            & u_e(y)\partial_xu_e^{(2)}+v_e^{(2)}u'_e(y)+\partial_xp_e^{(2)}=f_e^{(6)}, \\
            &u_e(y)\partial_xv_e^{(2)}+\partial_yp_e^{(2)}=g_e^{(6)},\\
            &\partial_xu_e^{(2)}+\partial_yv_e^{(2)}=0, \\
            &(u_e^{(2)},v_e^{(2)})(x,y)=(u_e^{(2)},v_e^{(2)})(x+2\pi, y),\\
            &v_e^{(2)}|_{y=1}=-v_p^{(2)}|_{\xi=0},~v_e^{(2)}|_{y=0}=-\hat{v}_p^{(2)}|_{\gamma=0},
        \end{aligned}
    \right .
    \end{align}
    where 
    \[
    \begin{aligned}
        &f_e^{(6)}=-u_e^{(1)}\partial_xu_e^{(1)}-v_e^{(1)}\partial_yu_e^{(1)}+\bigtriangleup u_e^{(1)}, \\
        &g_e^{(6)}=-u_e^{(1)}\partial_xv_e^{(1)}-v_e^{(1)}\partial_yv_e^{(1)}+\bigtriangleup v_e^{(1)},
    \end{aligned}
    \]
    and $(v_p^{(2)}, \hat{v}_p^{(2)})$ can be found in (\ref{equpk}) and (\ref{eqhupk}) respectively in next subsection. \\
    By eliminating $p_e^{(2)}$ from (\ref{eque6}), we obtain
    \[
    \left \{
    \begin{aligned}
        &-u_e(y)\triangle v_e^{(2)}+u''_e(y)v_e^{(2)}=\partial_yf_e^{(6)}-\partial_xg_e^{(6)}, \\
        &v_e^{(2)}(x, y)=v_e^{(2)}(x+2\pi, y), \\
        &v_e^{(2)}|_{y=1}=-v_p^{(2)}|_{\xi=0},~v_e^{(2)}|_{y=0}=-\hat{v}_p^{(2)}|_{\gamma=0}.
    \end{aligned}
    \right .
    \]
    Thus, we only need to show that $\int_{0}^{2\pi} \partial_yf_e^{(6)} \mathrm{d}x = 0$. 
    \[
    \begin{aligned}
        \int_{0}^{2\pi} \partial_yf_e^{(6)} \mathrm{d}x =&\int_{0}^{2\pi}( u_e^{(1)} \partial_{yy}v_e^{(1)}-v_e^{(1)}\partial_{y}^2u_e^{(1)}+\partial_{y}^3u_e^{(1)}) \mathrm{d}x \\
        =&-\int_{0}^{2\pi} \partial_{y}^2v_e^{(1)}\int_{0}^{x}\partial_yv_e^{(1)}(\bar{x},y)\mathrm{d}\bar{x}\mathrm{d}x+
        \int_{0}^{2\pi}v_e^{(1)}\int_{0}^{x}\partial_{y}^3v_e^{(1)}(\bar{x},y)\mathrm{d}\bar{x}\mathrm{d}x \\
        &-\int_{0}^{2\pi}\int_{0}^{x}\partial_y^4v_e^{(1)}(\bar{x},y)\mathrm{d}\bar{x}\mathrm{d}x+2\pi Z'''(y) \\
        =&0,
    \end{aligned}
    \]
    which indicates (\ref{add condition}). \\
    Furthermore, we can get that $\int_{0}^{2\pi} f_e^{(6)}\mathrm{d}x$ is a constant, denoted by $C_6:=\int_{0}^{2\pi} f_e^{(6)}\mathrm{d}x$.
\end{remark}

\end{proof}

\subsubsection{The solvability of boundary layer equations for $(u_p^{(1)}, v_p^{(\frac{4}{3})})$ and $(\hat{u}_p^{(1)}, \hat{v}_p^{(\frac{4}{3})})$.}
\begin{proposition} \label{Existence of boundary layer-1}
    The boundary layer equations (\ref{equp3}) and (\ref{eqhup3}) have a solution $ (u_p^{(1)}, v_p^{(\frac{4}{3})}) $ and $(\hat{u}_p^{(1)}, \hat{v}_p^{(\frac{4}{3})})$ which satisfies
    \begin{align} \label{p281}
    \begin{aligned}
       & \int_{-\infty}^{0} \int_{0}^{2\pi} |\partial_x^j\partial_{\xi}^m(u_p^{(1)}-A_1, v_p^{(\frac{4}{3})})|^2 \xi^{2l} \mathrm{d}x \mathrm{d}\xi \le C_{j,m,l}, \quad \forall j,m,l\in \mathbb{N}, \\
       &\int_{0}^{2\pi} v_p^{(\frac{4}{3})}(x, \xi) \mathrm{d}x=0, \quad \forall \xi \le 0,
    \end{aligned}
    \end{align}
    and
    \begin{align} \label{p282}
    \begin{aligned}
       & \int_{0}^{+\infty} \int_{0}^{2\pi} |\partial_x^j\partial_{\gamma}^m(\hat{u}_p^{(1)}-A_2, \hat{v}_p^{(\frac{4}{3})})|^2 \gamma^{2l} \mathrm{d}x \mathrm{d}\gamma \le C_{j,m,l}, \quad \forall j,m,l\in \mathbb{N}, \\
       &\int_{0}^{2\pi} \hat{v}_p^{(\frac{4}{3})}(x, \gamma) \mathrm{d}x=0, \quad \forall \gamma \ge 0,
    \end{aligned}
    \end{align}
\end{proposition}
\noindent where $A_1:= \lim\limits_{\xi \to -\infty} u_p^{(1)}(x, \xi),~~A_2:=\lim\limits_{\gamma \to +\infty} \hat{u}_p^{(1)}(x, \gamma)$ are two constants. \\
Furthermore,
\begin{align} \label{addition}
   u_p^{(1)}(x,-\gamma)=\hat{u}_p^{(1)}(x,\gamma),~~ v_p^{(\frac{4}{3})}(x, -\gamma)=-\hat{v}_p^{(\frac{4}{3})}(x, \gamma), ~~ \forall (x,\gamma)\in [0, 2\pi]\times [0,+\infty).
\end{align}
\begin{proof}
The well-posedness can be find in Proposition 4.6 of \cite{FPZ}. Now, we demonstrate that (\ref{addition}). In fact, for $(x,\bar{\gamma})\in [0, 2\pi]\times [0,+\infty)$, let
     $$\breve{u}_p^{(1)}(x, \bar{\gamma}):=\hat{u}_p^{(1)}(x,-\bar{\gamma}),~\breve{p}_p^{(\frac{4}{3})}(x, \bar{\gamma}):=\hat{p}_p^{(\frac{4}{3})}(x, -\bar{\gamma}), ~\breve{v}_p^{(\frac{4}{3})}(x, \bar{\gamma}):=-\hat{v}_p^{(\frac{4}{3})}(x, -\bar{\gamma}), ~$$
      then we obtain
    \[
    \left \{
    \begin{aligned}
        &-\alpha \bar{\gamma} \partial_x \breve{u}_p^{(1)}-\alpha \breve{v}_p^{(\frac{4}{3})}+\partial_x\breve{p}_p^{(\frac{4}{3})}-\partial_{\bar{\gamma} \bar{\gamma}} \breve{u}_p^{(1)}=0, \\
        &\partial_x \breve{u}_p^{(1)}+ \partial_{\bar{\gamma}}\breve{v}_p^{(\frac{4}{3})}=0, \\
        &\partial_{\bar{\gamma}}\breve{p}_p^{(\frac{4}{3})}=0, \\
        &(\breve{u}_p^{(1)}, \breve{v}_p^{(\frac{4}{3})})(x, \bar{\gamma}) = (\breve{u}_p^{(1)}, \breve{v}_p^{(\frac{4}{3})})(x+2\pi, \bar{\gamma}), \\
        &\breve{u}_p^{(1)}|_{\bar{\gamma} = 0} = -u_e^{(1)}|_{y=0}, \\
        &\lim\limits_{\bar{\gamma} \to -\infty} (\partial_{\bar{\gamma}}\breve{u}_p^{(1)}, \breve{v}_p^{(\frac{4}{3})}, \breve{p}_p^{(\frac{4}{3})}) = (0,0,0).
    \end{aligned}
    \right .
    \]
    Since $u_e^{(1)}(x, 0) = u_e^{(1)}(x, 1)$, there holds
     $$ u_p^{(1)}(x, \xi) = \breve{u}_p^{(1)}(x, \xi), ~~v_p^{(\frac{4}{3})}(x, \xi) = \breve{v}_p^{(\frac{4}{3})}(x, \xi), ~~(x,\xi)\in [0, 2\pi]\times [0,+\infty), $$
      that is  (\ref{addition}).
\end{proof}

\begin{remark}
    Since we want the solution of the boundary layer equations decaying fast at infinity, we define
    \[
    \tilde{u}_p^{(1)}:=u_p^{(1)}-A_1,~~\tilde{\hat{u}}_p^{(1)}:=\hat{u}_p^{(1)}-A_2,
    \]
    then we use $\tilde{u}_p^{(1)}$ and $\tilde{\hat{u}}_p^{(1)}$ to replace the original $u_p^{(1)}$ and $\hat{u}_p^{(1)}$ respectively. \\
    \indent We need to correct $u_e^{(1)}$ to $\tilde{u}_e^{(1)}$ such that
    \[
    \begin{aligned}
        &\tilde{u}_e^{(1)}|_{y=1}=u_e^{(1)}|_{y=1}+A_1, \\
        &\tilde{u}_e^{(1)}|_{y=0}=u_e^{(1)}|_{y=0}+A_2,
    \end{aligned}
    \]
    and $(\tilde{u}_e^{(1)},v_e^{(1)})$ is still a solution of the equations (\ref{eque3}). \\
    \indent Let
    \[
    \tilde{u}_e^{(1)}=u_e^{(1)}+Q^{(1)}(y),
    \]
    where $Q^{(1)}(y) \in C^{\infty}([0,1])$ and it satisfies $Q^{(1)}(1)=A_1$, $Q^{(1)}(0)=A_2$. Then we have
    \begin{enumerate}
        \item[$\bullet$] $(\tilde{u}_e^{(1)},v_e^{(1)})$ is still a solution of the equations (\ref{eque3});
        \item[$\bullet$] $\tilde{u}_e^{(1)}|_{y=1}+\tilde{u}_p^{(1)}|_{\xi=0}$=0,~~$\tilde{u}_e^{(1)}|_{y=0}+\tilde{\hat{u}}_p^{(1)}|_{\gamma=0}$=0.
    \end{enumerate}
    Then we use $\tilde{u}_p^{(1)}$, $\tilde{\hat{u}}_p^{(1)}$, $\tilde{u}_e^{(1)}$ to be the new $u_p^{(1)}$, $\hat{u}_p^{(1)}$, $u_e^{(1)}$ to enter the subsequent round iteration.
\end{remark}

\subsection{Construction of higher order terms}
\subsubsection{Equation for $(u_e^{(\frac{k}{3})}, v_e^{(\frac{k}{3})})$} For $k \ge 4$, by substituting the Euler expansion (\ref{Euelr expansion}) into (\ref{NS equation}) and collecting the $\epsilon^{\frac{k}{3}}$ order terms, we have the following steady Euler equations
\begin{align} \label{equek}
    \left \{
    \begin{aligned}
        &u_e(y)\partial_xu_e^{(\frac{k}{3})}+u'_e(y)v_e^{(\frac{k}{3})}+\partial_x p_e^{(\frac{k}{3})} = f_e^{(k)}, \\
        &u_e(y)\partial_xv_e^{(\frac{k}{3})}+\partial_yp_e^{(\frac{k}{3})} = g_e^{(k)}, \\
        &\partial_xu_e^{(\frac{k}{3})} + \partial_yv_e^{(\frac{k}{3})} = 0, \\
        &(u_e^{(\frac{k}{3})}, v_e^{(\frac{k}{3})})(x, y)=(u_e^{(\frac{k}{3})}, v_e^{(\frac{k}{3})})(x+2\pi, y),\\
        &v_e^{(\frac{k}{3})}|_{y=1}=-v_p^{(\frac{k}{3})}|_{\xi = 0},~v_e^{(\frac{k}{3})}|_{y=0}=-\hat{v}_p^{(\frac{k}{3})}|_{\gamma=0},
    \end{aligned}
    \right .
\end{align}
where
\[
\begin{aligned}
    &f_e^{(k)}:=-\sum\limits_{i=3}^{k-3}u_e^{(\frac{i}{3})}\partial_xu_e^{(\frac{k-i}{3})}-\sum\limits_{i=3}^{k-3}v_e^{(\frac{i}{3})}\partial_yu_e^{(\frac{k-i}{3})}+\bigtriangleup u_e^{(\frac{k-3}{3})}, \\
    &g_e^{(k)}:=-\sum\limits_{i=3}^{k-3} u_e^{(\frac{i}{3})}\partial_xv_e^{(\frac{k-i}{3})}-\sum\limits_{i=3}^{k-3} v_e^{(\frac{i}{3})}\partial_yv_e^{(\frac{k-i}{3})}+\bigtriangleup v_e^{(\frac{k-3}{3})},
\end{aligned}
\]
and $(v_p^{(\frac{k}{3})}, \hat{v}_p^{(\frac{k}{3})})$ is defined in (\ref{equpk}) and (\ref{eqhupk}) respectively in next subsection.
\subsubsection{Equation for $(u_p^{(\frac{k}{3})}, v_p^{(\frac{k+1}{3})}, p_p^{(\frac{k+1}{3})})$} For $k \ge 4$, by substituting the upper boundary expansion (\ref{upper boundary layer expansion}) into (\ref{NS equation}) and collecting the $\epsilon^{\frac{k}{3}}/\epsilon^{\frac{k+1}{3}}$ order terms, we have the following steady boundary layer equations
\begin{align} \label{equpk}
    \left \{
    \begin{aligned}
        &-\alpha \xi \partial_x u_p^{(\frac{k}{3})}-\alpha v_p^{(\frac{k+1}{3})}+\partial_xp_p^{(\frac{k+1}{3})}-\partial_{\xi \xi}u_p^{(\frac{k}{3})} = f_{p,k}(x, \xi), \\
        &\partial_{\xi} p_p^{(\frac{k+1}{3})}=g_{p,k}(x, \xi), \\
        &\partial_xu_p^{(\frac{k}{3})}+\partial_{\xi}v_p^{(\frac{k+1}{3})} = 0, \\
        &(u_p^{(\frac{k}{3})}, v_p^{(\frac{k+1}{3})})(x, \xi) = (u_p^{(\frac{k}{3})}, v_p^{(\frac{k+1}{3})})(x+2\pi, \xi), \\
        &u_p^{(\frac{k}{3})}|_{\xi = 0}=-u_e^{(\frac{k}{3})}|_{y=1}, \\
        &\lim\limits_{\xi \to -\infty}(v_p^{(\frac{k+1}{3})}, p_p^{(\frac{k+1}{3})})=(0,0,0).
    \end{aligned}
    \right .
\end{align}
Here,
\[
\begin{aligned}
    f_{p,k}:=&\alpha \xi^2 \partial_xu_p^{(\frac{k-1}{3})}+2\alpha \xi v_p^{(\frac{k}{3})}+\partial_{xx}u_p^{(\frac{k-2}{3})}-\sum\limits_{j=3}^{k-2}\partial_x u_p^{(\frac{k-j+1}{3})}\sum\limits_{i=3}^{j}\frac{\xi^{j-i}}{(j-i)!}\partial_y^{j-i}u_e^{(\frac{i}{3})}(x,1)  \\
    &-\sum\limits_{j=3}^{k-2}u_p^{(\frac{k-j+1}{3})}\sum\limits_{i=3}^{j}\frac{\xi^{j-i}}{(j-i)!}\partial_x\partial_y^{j-i}u_e^{(\frac{i}{3})}(x,1)\\
    &-\sum\limits_{j=3}^{k-2}u_p^{(\frac{j}{3})}\partial_xu_p^{(\frac{k-j+1}{3})}-\sum\limits_{j=3}^{k-1}\partial_{\xi}u_p^{(\frac{k-j+2}{3})}\sum\limits_{i=3}^{j}\frac{\xi^{j-i}}{(j-i)!}\partial_y^{j-i}v_e^{(\frac{i}{3})}(x,1) \\
    &-\sum\limits_{j=3}^{k-2}v_p^{(\frac{k-j+1}{3})}\sum\limits_{i=3}^{j}\frac{\xi^{j-i}}{(j-i)!}\partial_y^{j-i+1}u_e^{(\frac{i}{3})}(x,1)-\sum\limits_{j=3}^{k-1}v_p^{(\frac{j}{3})}\partial_{\xi}u_p^{(\frac{k-j+2}{3})},    
\end{aligned}
\]
and
\[
\begin{aligned}
    g_{p,k}:=&\alpha \xi \partial_xv_p^{(\frac{k-1}{3})}+\alpha \xi^2\partial_xv_p^{(\frac{k-2}{3})}+\partial_{xx}v_p^{(\frac{k-3}{3})}+\partial_{\xi \xi}v_p^{(\frac{k-1}{3})} \\
    &-\sum\limits_{j=3}^{k-3}\partial_xv_p^{(\frac{k-j}{3})}\sum\limits_{i=3}^{j}\frac{\xi^{j-i}}{(j-i)!}\partial_y^{j-i}u_e^{(\frac{i}{3})}(x,1) \\
    &-\sum\limits_{j=3}^{k-3}u_p^{(\frac{k-j}{3})}\sum\limits_{i=3}^{j}\frac{\xi^{j-i}}{(j-i)!}\partial_x\partial_y^{j-i}v_e^{(\frac{i}{3})}(x,1)-\sum\limits_{j=3}^{k-3}u_p^{(\frac{j}{3})}\partial_xv_p^{(\frac{k-j}{3})} \\
    &-\sum\limits_{j=3}^{k-2}\partial_{\xi} v_p^{(\frac{k-j+1}{3})}\sum\limits_{i=3}^{j}\frac{\xi^{j-i}}{(j-i)!}\partial_y^{j-i}v_e^{(\frac{i}{3})}(x,1) \\
    &-\sum\limits_{j=3}^{k-3} v_p^{(\frac{k-j}{3})}\sum\limits_{i=3}^{j}\frac{\xi^{j-i}}{(j-i)!}\partial_y^{j-i+1}v_e^{(\frac{i}{3})}(x, 1)-\sum\limits_{j=3}^{k-2}v_p^{(\frac{j}{3})}\partial_{\xi} v_p^{(\frac{k-j+1}{3})}.
\end{aligned}
\]

\subsubsection{Equation for $(\hat{u}_p^{(\frac{k}{3})}, \hat{v}_p^{(\frac{k+1}{3})}, \hat{p}_p^{(\frac{k+1}{3})})$} For $k \ge 4$, by substituting the lower boundary expansion (\ref{lower boundary layer expansion}) into (\ref{NS equation}) and collecting the $\epsilon^{\frac{k}{3}}/\epsilon^{\frac{k+1}{3}}$ order terms, we have the following steady boundary layer equations
\begin{align} \label{eqhupk}
    \left \{
    \begin{aligned}
        &\alpha \gamma \partial_x \hat{u}_p^{(\frac{k}{3})}+\alpha \hat{v}_p^{(\frac{k+1}{3})}+\partial_x\hat{p}_p^{(\frac{k+1}{3})}-\partial_{\gamma \gamma}\hat{u}_p^{(\frac{k}{3})} = \hat{f}_{p,k}(x, \gamma), \\
        &\partial_{\gamma} \hat{p}_p^{(\frac{k+1}{3})}=\hat{g}_{p,k}(x, \gamma), \\
        &\partial_x\hat{u}_p^{(\frac{k}{3})}+\partial_{\gamma}\hat{v}_p^{(\frac{k+1}{3})} = 0, \\
        &(\hat{u}_p^{(\frac{k}{3})}, \hat{v}_p^{(\frac{k+1}{3})})(x, \gamma) = (\hat{u}_p^{(\frac{k}{3})}, \hat{v}_p^{(\frac{k+1}{3})})(x+2\pi, \gamma), \\
        &\hat{u}_p^{(\frac{k}{3})}|_{\gamma = 0}=-u_e^{(\frac{k}{3})}|_{y=0}, \\
        &\lim\limits_{\gamma \to +\infty}(\hat{v}_p^{(\frac{k+1}{3})}, \hat{p}_p^{(\frac{k+1}{3})})=(0,0).
    \end{aligned}
    \right .
\end{align}
Here,
\[
\begin{aligned}
    \hat{f}_{p,k}:=&\alpha \gamma^2 \partial_x\hat{u}_p^{(\frac{k-1}{3})}+2\alpha \gamma \hat{v}_p^{(\frac{k}{3})}+\partial_{xx}\hat{u}_p^{(\frac{k-2}{3})}-\sum\limits_{j=3}^{k-2}\partial_x \hat{u}_p^{(\frac{k-j+1}{3})}\sum\limits_{i=3}^{j}\frac{\gamma^{j-i}}{(j-i)!}\partial_y^{j-i}u_e^{(\frac{i}{3})}(x,0)  \\
    &-\sum\limits_{j=3}^{k-2}\hat{u}_p^{(\frac{k-j+1}{3})}\sum\limits_{i=3}^{j}\frac{\gamma^{j-i}}{(j-i)!}\partial_x\partial_y^{j-i}u_e^{(\frac{i}{3})}(x,0)\\
    &-\sum\limits_{j=3}^{k-2}\hat{u}_p^{(\frac{j}{3})}\partial_x\hat{u}_p^{(\frac{k-j+1}{3})}
    -\sum\limits_{j=3}^{k-1}\partial_{\gamma}\hat{u}_p^{(\frac{k-j+2}{3})}\sum\limits_{i=3}^{j}\frac{\gamma^{j-i}}{(j-i)!}\partial_y^{j-i}v_e^{(\frac{i}{3})}(x,0) \\
    &-\sum\limits_{j=3}^{k-2}\hat{v}_p^{(\frac{k-j+1}{3})}\sum\limits_{i=3}^{j}\frac{\gamma^{j-i}}{(j-i)!}\partial_y^{j-i+1}u_e^{(\frac{i}{3})}(x,0)-\sum\limits_{j=3}^{k-1}\hat{v}_p^{(\frac{j}{3})}\partial_{\gamma}\hat{u}_p^{(\frac{k-j+2}{3})},    
\end{aligned}
\]
and
\[
\begin{aligned}
    \hat{g}_{p,k}:=&-\alpha \gamma \partial_x\hat{v}_p^{(\frac{k-1}{3})}+\alpha \gamma^2\partial_x\hat{v}_p^{(\frac{k-2}{3})}+\partial_{xx}\hat{v}_p^{(\frac{k-3}{3})}+\partial_{\gamma \gamma}\hat{v}_p^{(\frac{k-1}{3})} \\
    &-\sum\limits_{j=3}^{k-3}\partial_x\hat{v}_p^{(\frac{k-j}{3})}\sum\limits_{i=3}^{j}\frac{\gamma^{j-i}}{(j-i)!}\partial_y^{j-i}u_e^{(\frac{i}{3})}(x,0) \\
    &-\sum\limits_{j=3}^{k-3}\hat{u}_p^{(\frac{k-j}{3})}\sum\limits_{i=3}^{j}\frac{\gamma^{j-i}}{(j-i)!}\partial_x\partial_y^{j-i}v_e^{(\frac{i}{3})}(x,0)-\sum\limits_{j=3}^{k-3}\hat{u}_p^{(\frac{j}{3})}\partial_x\hat{v}_p^{(\frac{k-j}{3})} \\
    &-\sum\limits_{j=3}^{k-2}\partial_{\gamma} \hat{v}_p^{(\frac{k-j+1}{3})}\sum\limits_{i=3}^{j}\frac{\gamma^{j-i}}{(j-i)!}\partial_y^{j-i}v_e^{(\frac{i}{3})}(x,0) \\
    &-\sum\limits_{j=3}^{k-3} \hat{v}_p^{(\frac{k-j}{3})}\sum\limits_{i=3}^{j}\frac{\gamma^{j-i}}{(j-i)!}\partial_y^{j-i+1}v_e^{(\frac{i}{3})}(x, 0)-\sum\limits_{j=3}^{k-2}\hat{v}_p^{(\frac{j}{3})}\partial_{\gamma} \hat{v}_p^{(\frac{k-j+1}{3})}.
\end{aligned}
\]

Direct computer gives that $\hat{f}_{p,k}(x, -\xi)=f_{p,k}(x, \xi),~~-\hat{g}_{p,k}(x,-\xi)=g_{p,k}(x, \xi)$.
\subsubsection{The solvability of $(u_e^{(\frac{k}{3})}, v_e^{(\frac{k}{3})})$, $(u_p^{(\frac{k}{3})}, v_p^{(\frac{k+1}{3}})$ and $(\hat{u}_p^{(\frac{k}{3})}, \hat{v}_p^{(\frac{k+1}{3})})$} Since their solvability analysis is analogous to Proposition \ref{existence of Euler-1} and Proposition \ref{Existence of boundary layer-1}, we omit the details and proceed directly to the results. 
\begin{proposition} \label{Existence of higher order Euler}
    The Euler equations (\ref{equek}) has a solution $ (u_e^{(\frac{k}{3})}, v_e^{(\frac{k}{3})}) $ which satisfies 
    \begin{align} \label{p284}
    \begin{aligned}
        &\|\partial_x^j \partial_y^m (u_e^{(\frac{k}{3})}, v_e^{(\frac{k}{3})})\|_{L^2} \leq C_{j,m},~~\forall j,m \geq 0, \\
        &\int_{0}^{2\pi} v_e^{(\frac{k}{3})}(x,y) \mathrm{d}x = 0, ~~\forall y \in [0,1],\\
        &u_e^{(\frac{k}{3})}(x, y) = u_e^{(\frac{k}{3})}(x, 1-y),~~v_e^{(\frac{k}{3})}(x, y) = -v_e^{(\frac{k}{3})}(x, 1-y).
    \end{aligned}
    \end{align}
\end{proposition}

\begin{remark}
    Here, we provide the construction of $ (u_e^{(\frac{k}{3})},v_e^{(\frac{k}{3})},p_e^{(\frac{k}{3})}). $ We make an inductive hypothesis that for $i=3,...,k-1$,Propsition \ref{Existence of higher order Euler} and Proposition \ref{Existence of higher order boundary layer}  hold. Moreover, we assume that
    $$\frac{1}{2\pi}\int_{0}^{2\pi}f_e^{(i)}\mathrm{d}x=C_i,i=6,...,k+2.$$
    Eliminating $p_e^{(\frac{k}{3})}$ from (\ref{equpk}), we obtain that
    \[
    \left \{
    \begin{aligned}
    	&-u_e(y)\triangle v_e^{(\frac{k}{3})}+u''_e(y)v_e^{(\frac{k}{3})}=\partial_yf_e^{(k)}-\partial_xg_e^{(k)}, \\
    	&v_e^{(\frac{k}{3})}(x, y)=v_e^{(\frac{k}{3})}(x+2\pi, y), \\
    	&v_e^{(\frac{k}{3})}|_{y=1}=-v_p^{(\frac{k}{3})}|_{\xi=0},~v_e^{(\frac{k}{3})}|_{y=0}=-\hat{v}_p^{(\frac{k}{3})}|_{\gamma=0}.
    \end{aligned}
    \right .
    \]
    Analogously to Proposition \ref{existence of Euler-1}, we can prove the well-posedness. Then we define
    \[u_e^{(\frac{k}{3})}(x,y)=-\int_{0}^{x}\partial_yv_e^{(\frac{k}{3})}(\bar{x},y)\mathrm{d}\bar{x}+Z_k(y).\]
    where $Z_k(y)$ satisifies $Z_k(1-y)=Z_k(y)$ and
    \[
    \begin{aligned}
    &2\pi Z_k'''(y)-\sum_{i=3}^{k}\int_{0}^{2\pi}\partial_{y}^2v_e^{(\frac{k+3-i}{3})}(x,y)\int_{0}^{x}\partial_yv_e^{(\frac{i}{3})}(\bar{x},y)\mathrm{d}\bar{x}\mathrm{d}x \\
    +&\sum_{i=3}^{k}\int_{0}^{2\pi}v_e^{(\frac{i}{3})}(x,y)\int_{0}^{x}\partial_{y}^3v_e^{(\frac{k+3-i}{3})}(\bar{x},y)\mathrm{d}\bar{x}\mathrm{d}x-\int_{0}^{2\pi}\int_{0}^{x}\partial_y^4v_e^{(\frac{k}{3})}(\bar{x},y)\mathrm{d}\bar{x}\mathrm{d}x=0.
    \end{aligned}
    \]
    Analogously to Remark \ref{remark2.2}, we can prove that $ \frac{1}{2\pi}\int_{0}^{2\pi}f_e^{(k+3)}\mathrm{d}x=C_{k+3}. $ Here, we omit the detailed proof. After obtaining $ (u_e^{(\frac{k}{3})}, v_e^{(\frac{k}{3})}) $, we construct $ p_e^{(\frac{k}{3})} $ as follows
    \[
    p_e^{(\frac{k}{3})}=W_k(y)-u_e(y)u_e^{(\frac{k}{3})}-u'_e(y)\int_{0}^{x}v_e^{(\frac{k}{3})}(\bar{x}, y)\mathrm{d}\bar{x}+\int_{0}^{x}f_e^{(k)}(\bar{x}, y)\mathrm{d}\bar{x}.
    \]
    So, it satisfies
    \[    u_e(y)\partial_xu_e^{(\frac{k}{3})}+v_e^{(\frac{k}{3})}u'_e(y)+\partial_xp_e^{(\frac{k}{3})}=f_e^{(k)}.
    \]
    We let $W_k(y)$ satisfy
    \[
    W_k'(y)-u'_e(y)u_e^{(\frac{k}{3})}(0, y)+u_e(y)\partial_xv_e^{(\frac{k}{3})}(0, y)-u_e(y)\partial_yu_e^{(\frac{k}{3})}(0, y)=g_e^{(k)}(0, y).
    \]
    Combining the properties of $ (u_e^{(\frac{k}{3})}, v_e^{(\frac{k}{3})}) $, it follows that
    \[
    u_e(y)\partial_xv_e^{(\frac{k}{3})}+\partial_yp_e^{(\frac{k}{3})}=g_e^{(k)}.
    \]
    At last, we have for $ k\ge 6 $
    \begin{equation} \label{pek}
    p_e^{(\frac{k}{3})}=\tilde{p}_e^{(\frac{k}{3})}+C_kx,
    \end{equation}
    where $ \tilde{p}_e^{(\frac{k}{3})} $ is periodic in x.
\end{remark}

\begin{proposition} \label{Existence of higher order boundary layer}
    The boundary layer equations (\ref{equpk}) and (\ref{eqhupk}) have a solution $ (u_p^{(\frac{k}{3})}, v_p^{(\frac{k+1}{3})}) $ and $(\hat{u}_p^{(\frac{k}{3})}, \hat{v}_p^{(\frac{k+1}{3})})$ which satisfies
    \begin{align} \label{p285}
    \begin{aligned}
       & \int_{-\infty}^{0} \int_{0}^{2\pi} |\partial_x^j\partial_{\xi}^m(u_p^{(\frac{k}{3})}-A_{1,\frac{k}{3}}, v_p^{(\frac{k+1}{3})})|^2 \xi^{2l} \mathrm{d}x \mathrm{d}\xi \le C_{j,m,l}, \quad \forall j,m,l\in \mathbb{N}, \\
       &\int_{0}^{2\pi} v_p^{(\frac{k+1}{3})}(x, \xi) \mathrm{d}x=0, \quad \forall \xi \le 0,
    \end{aligned}
    \end{align}
    and
    \begin{align} \label{p286}
    \begin{aligned}
       & \int_{0}^{+\infty} \int_{0}^{2\pi} |\partial_x^j\partial_{\gamma}^m(\hat{u}_p^{(\frac{k}{3})}-A_{2, \frac{k}{3}}, \hat{v}_p^{(\frac{k+1}{3})})|^2 \gamma^{2l} \mathrm{d}x \mathrm{d}\gamma \le C_{j,m,l}, \quad \forall j,m,l\in \mathbb{N}, \\
       &\int_{0}^{2\pi} \hat{v}_p^{(\frac{k+1}{3})}(x, \gamma) \mathrm{d}x=0, \quad \forall \gamma \ge 0,
    \end{aligned}
    \end{align}
    \noindent where $A_{1, \frac{k}{3}}=\lim\limits_{\xi \to -\infty} u_p^{(\frac{k}{3})}(x, \xi),~~A_{2, \frac{k}{3}}=\lim\limits_{\gamma \to +\infty} \hat{u}_p^{(\frac{k}{3})}(x, \gamma)$ are two constants. \\
   Furthermore, for $(x,\xi)\in [0,2\pi]\times (-\infty,0]$, there holds
     $$u_p^{(\frac{k+1}{3})}(x, \xi)=\hat{u}_p^{(\frac{k+1}{3})}(x,-\xi), ~~v_p^{(\frac{k+1}{3})}(x, \xi)=-\hat{v}_p^{(\frac{k+1}{3})}(x, -\xi).$$
\end{proposition}

\begin{remark}
    Since we want the solution of the boundary layer equations decaying fast at infinity, we define for $k\ge 4$,
    \[
    \tilde{u}_p^{(\frac{k}{3})}:=u_p^{(\frac{k}{3})}-A_{1,\frac{k}{3}},~~\tilde{\hat{u}}_p^{(\frac{k}{3})}:=\hat{u}_p^{(\frac{k}{3})}-A_{2,\frac{k}{3}},
    \]
    then we use $\tilde{u}_p^{(\frac{k}{3})}$ and $\tilde{\hat{u}}_p^{(\frac{k}{3})}$ to replace the original $u_p^{(\frac{k}{3})}$ and $\hat{u}_p^{(\frac{k}{3})}$ respectively. \\
    \indent We need to correct $u_e^{(\frac{k}{3})}$ to $\tilde{u}_e^{(\frac{k}{3})}$ such that
    \[
    \begin{aligned}
        &\tilde{u}_e^{(\frac{k}{3})}|_{y=1}=u_e^{(\frac{k}{3})}|_{y=1}+A_{1,\frac{k}{3}}, \\
        &\tilde{u}_e^{(\frac{k}{3})}|_{y=0}=u_e^{(\frac{k}{3})}|_{y=0}+A_{2,\frac{k}{3}},
    \end{aligned}
    \]
    and $(\tilde{u}_e^{(\frac{k}{3})},v_e^{(\frac{k}{3})})$ is still a solution of the equations (\ref{equek}). \\
    \indent Let
    \[
    \tilde{u}_e^{(\frac{k}{3})}=u_e^{(\frac{k}{3})}+Q^{(k)}(y),
    \]
    where $Q^{(k)}(y) \in C^{\infty}([0,1])$ and it satisfies $Q^{(k)}(1)=A_{1,\frac{k}{3}}$, $Q^{(k)}(0)=A_{2,\frac{k}{3}}$. Then we have
    \begin{enumerate}
        \item[$\bullet$] $(\tilde{u}_e^{(\frac{k}{3})},v_e^{(\frac{k}{3})})$ is still a solution of the equations (\ref{equek});
        \item[$\bullet$] $\tilde{u}_e^{(\frac{k}{3})}|_{y=1}+\tilde{u}_p^{(\frac{k}{3})}|_{\xi=0}$=0,~~$\tilde{u}_e^{(\frac{k}{3})}|_{y=0}+\tilde{\hat{u}}_p^{(\frac{k}{3})}|_{\gamma=0}$=0.
    \end{enumerate}
    Then we use $\tilde{u}_p^{(\frac{k}{3})}$, $\tilde{\hat{u}}_p^{(\frac{k}{3})}$, $\tilde{u}_e^{(\frac{k}{3})}$ to be the new $u_p^{(\frac{k}{3})}$, $\hat{u}_p^{(\frac{k}{3})}$, $u_e^{(\frac{k}{3})}$ to enter the subsequent round iteration.
\end{remark}

\subsection{Construction of approximate solutions}  \label{Approximate solutions}
In this subsection, we construct an approximate solution of the system (\ref{NS equation}). Let $ \chi(y) \in C_c^{\infty}([0, +\infty)) $ be a cut-off function satisfies
\[
\chi(y)+\chi(1-y)=1,~~~~
\chi(y) = \left \{
\begin{aligned}
    &1, ~~y \in \Big[0, \frac{1}{4}\Big), \\
    &0, ~~y \in \Big[\frac{3}{4},1\Big]
\end{aligned}
\right .
\]
Set
\[
\begin{aligned}
    &\tilde{u}_p^a:=(1-\chi(y))\sum\limits_{k=3}^{15}\epsilon^{\frac{k}{3}}u_p^{(\frac{k}{3})}+\chi(y)\sum\limits_{k=3}^{15}\epsilon^{\frac{k}{3}}\hat{u}_p^{(\frac{k}{3})}:=(1-\chi(y))u_p^a+\chi(y)\hat{u}_p^a, \\
    &\tilde{v}_p^a:=(1-\chi(y))\sum\limits_{k=3}^{15}\epsilon^{\frac{k+1}{3}}v_p^{(\frac{k+1}{3})}+\chi(y)\sum\limits_{k=3}^{15}\epsilon^{\frac{k+1}{3}}\hat{v}_p^{(\frac{k+1}{3})}:=(1-\chi(y))v_p^a+\chi(y)\hat{v}_p^a, \\
    &\tilde{p}_p^a:=(1-\chi(y))^2\sum\limits_{k=3}^{15}\epsilon^{\frac{k+1}{3}}p_p^{(\frac{k+1}{3})}
    +\chi^2(y)\sum\limits_{k=3}^{15}\epsilon^{\frac{k+1}{3}}\hat{p}_p^{(\frac{k+1}{3})}:=(1-\chi(y))^2p_p^a+\chi^2(y)\hat{p}_p^a,
\end{aligned}
\]
and
\[
\begin{aligned}
    u_e^a := u_e(y)+\sum\limits_{k=3}^{15}\epsilon^{\frac{k}{3}}u_e^{(\frac{k}{3})}, ~~
    v_e^a := \sum\limits_{k=3}^{15}\epsilon^{\frac{k}{3}}v_e^{(\frac{k}{3})}, ~~
    p_e^a := \sum\limits_{k=3}^{15}\epsilon^{\frac{k}{3}}p_e^{(\frac{k}{3})}.
\end{aligned}
\]
We construct an approximate solution $(u^a, v^a, p^a)$ by
\begin{align} \label{construction approximate solution}
\begin{aligned}
    &u^a(x, y):=u_e^a+\tilde{u}_p^a+\epsilon^5 h(x, y), \\[5pt]
    &v^a(x, y):=v_e^a+\tilde{v}_p^a, \\[5pt]
    &p^a(x, y):=p_e^a+\tilde{p}_p^a=\tilde{p}^a+x\sum\limits_{k=3}^{15}\epsilon^{\frac{k}{3}}C_k,
\end{aligned}
\end{align}
where $\tilde{p}^a$ is periodic in $x$ and the corrector $h(x, y)$ will be given in Appendix, which satisfies
\[
h(x, y)=h(x, 1-y),~~h(x, 0) =h(x, 1) = 0,~~\|\partial_x^j\partial_y^kh\|_{L^2} \leq C_{j,k}\epsilon^{-\frac{k}{3}},
\]
and makes $(u^a, v^a)$ be divergence-free
\[
\partial_x u^a+\partial_yv^a =0.
\]
Combining the Proposition \ref{existence of Euler-1}-Proposition \ref{Existence of higher order boundary layer}, we obtain
\begin{align} \label{uasymmetric}
    u^a(x, y)=u^a(x, 1-y),~~v^a(x, y)=-v^a(x, 1-y)
\end{align}
and
\begin{align} \label{estimate on approximate solution}
&\|\partial_x^j\partial_y^k\big(u_e^a-u_e(y)\big)\|_{L^{\infty}}\leq C_{j,k}\epsilon,~~\|\partial_x^j\partial_y^kv_e^a\|_{L^{\infty}}\leq C_{j,k}\epsilon, \nonumber\\
&\|\xi^l\partial_x^j\partial_y^ku_p^a\|_{L^{\infty}}\leq C_{j,k,l}\epsilon, ~~ \|\xi^l\partial_x^j\partial_y^kv_p^a\|_{L^{\infty}}\leq C_{j,k,l}\epsilon^{\frac{4}{3}}, \nonumber\\
&\|\gamma^l\partial_x^j\partial_y^k\hat{u}_p^a\|_{L^{\infty}}\leq C_{j,k,l}\epsilon, ~~ \|\gamma^l\partial_x^j\partial_y^k\hat{v}_p^a\|_{L^{\infty}}\leq C_{j,k,l}\epsilon^{\frac{4}{3}}. 
\end{align}
Finally, set
\[
\begin{aligned}
    &R_u^a:=u^a\partial_xu^a+v^a\partial_yu^a+\partial_xp^a-\epsilon \triangle u^a-\epsilon F_u, \\
    &R_v^a:=u^a\partial_xv^a+v^a\partial_yv^a+\partial_yp^a-\epsilon \triangle v^a-\epsilon F_v,
\end{aligned}
\]
then there holds
\[
\|R_u^a\|_{L^2}+\|\partial_xR_u^a\|_{L^2}\leq C\epsilon^5,~~\|R_v^a\|_{L^2}+\|\partial_xR_v^a\|_{L^2}\leq C\epsilon^5,
\]
and $(u^a, v^a, p^a)$ satisfies
\[
\left \{
\begin{aligned}
    &u^a\partial_xu^a+v^a\partial_yu^a+\partial_xp^a-\epsilon\triangle u^a=\epsilon F_u+R_u^a, ~(x,y)\in \mathbb{T} \times [0, 1],\\
    &u^a\partial_xv^a+v^a\partial_yv^a+\partial_yp^a-\epsilon\triangle v^a=\epsilon F_v+R_v^a, ~(x,y)\in \mathbb{T} \times [0, 1],\\
    &\partial_xu^a+\partial_yv^a=0, ~(x,y)\in \mathbb{T} \times [0, 1],\\
    &u^a(x, y)=u^a(x+2\pi, y), v^a(x, y)=v^a(x+2\pi, y), ~(x,y)\in \mathbb{T} \times [0, 1],\\
    &u^a(x, 1)=0, v^a(x, 1)=0, ~ x\in [0, 2\pi],\\
    &u^a(x, 0)=0, v^a(x, 0)=0, ~ x\in [0, 2\pi], \\
    &u^a(x, y)=u^a(x, 1-y),~v^a(x, y)=-v^a(x, 1-y).
\end{aligned}
\right .
\]

\section{Linear stability estimate for error equations} \label{sec3}
In this section, we derive the error equations and establish the linear stability estimate for the error equations.
\subsection{Error equations}In this subsection, the error equation is derived in both its velocity form and stream function form.
\subsubsection{Velocity formulation of error equation}Set the error by $ u:=u^{\varepsilon}-u^a,  v:=v^{\varepsilon}-v^a, p:=p^{\varepsilon}-p^a$. Then there holds
\begin{align} \label{error}
    \left \{
    \begin{aligned}
        &-\epsilon\triangle u+\partial_xp+S_u=R_u, ~(x,y)\in \mathbb{T} \times [0, 1],\\
        &-\epsilon\triangle v+\partial_yp+S_v=R_v, ~(x,y)\in \mathbb{T} \times [0, 1],\\
        &\partial_xu+\partial_yv=0, ~(x,y)\in \mathbb{T} \times [0, 1],\\
        &u(x, y)=u(x+2\pi, y), v(x, y)= v(x+2\pi, y), ~(x,y)\in \mathbb{T} \times [0, 1],\\
        &u(x, 1)=0, v(x, 1)=0, ~ x\in [0, 2\pi],\\
        &u(x, 0)=0, v(x, 0)=0, ~ x\in [0, 2\pi],
    \end{aligned}
    \right .\
\end{align}
where the convection terms are given by
\[
\begin{aligned}
    &S_u:=u^a\partial_xu+v^a\partial_yu+u\partial_xu^a+v\partial_yu^a,
    \\    &S_v:=u^a\partial_xv+v^a\partial_yv+u\partial_xv^a+v\partial_yv^a,
\end{aligned}
\]
and the remainders are defined by
\[
\begin{aligned}
    R_u:=-R_u^a-u\partial_xu-v\partial_yu, ~~
    R_v:=-R_v^a-u\partial_xv-v\partial_yv.
\end{aligned}
\]

\begin{Remark}We seek the pressure correction $p=p^\epsilon-p^a$ in the
class of $2\pi$-periodic functions in $x$. Thus, the exact
pressure and the approximate pressure have the same linear
part:
\[
    p^\epsilon
    =
    \bigl(\widetilde p^a+p\bigr)+\Pi_\epsilon x,
    \qquad
    \Pi_\epsilon
    =
    \sum_{k=3}^{15}\epsilon^{k/3}C_k,
\]
where $\widetilde p^a+p$ is $2\pi$-periodic in $x$.
\end{Remark}

\subsubsection{Stream function formulation of error equation} Here we give a reformulation of system (\ref{error}) in the stream function. \\
\indent 
Taking
\[
\phi(x, y) = \int_{1}^{y} u(x, \bar{y}) \mathrm{d}\bar{y},
\]
then $\phi$ is $2\pi$-periodic function in $x$ and
\[
\bigtriangledown^{\bot} \phi = (\partial_y, -\partial_x)\phi = (u, v).
\]
The boundary condition  $v|_{y=0}=0$ indicates that
\[
\partial_x \phi(x,0)=0,
\]
which indicates that $\phi(x,0)=$ constant. Thus, the stream-function formulation of error equation is
\begin{align} \label{vorticity equation}
    \left \{
    \begin{aligned}
        &-\epsilon \triangle^2 \phi+(u^a\partial_x+v^a\partial_y)w+(u\partial_x+v\partial_y)\omega^a=\partial_xR_v-\partial_yR_u, \\[3pt]
        & u=\partial_y \phi,~~ v=-\partial_x \phi, ~~ w=\triangle \phi, \\[3pt]
        &\phi|_{y=1}=0,~~\phi(x,0)=constant,~~\partial_y \phi|_{y=0,1}=0,~~\phi(x,y)=\phi(x+2\pi, y),
    \end{aligned}
    \right .
\end{align}
where $\omega^a:= \partial_xv^a-\partial_yu^a$, and we have used the following fact
\[
\partial_xS_v-\partial_yS_u=(u^a\partial_x+v^a\partial_y)\omega+(u\partial_x+v\partial_y)\omega^a.
\]

\subsection{Linear stability estimate of error equation (\ref{error})} Now, we provide the main estimate for the linearized Navier-Stokes system (\ref{error}). Define the $x$-mean of a function $f(x,y)$ by
    \[
    f_0 = \frac{1}{2\pi} \int_{0}^{2\pi} f(x, y)\mathrm{d}x
    \]
    and $\bar{f}=f-f_0$.
    
    We begin by collecting several estimates on $(u^a, v^a)$, which will be needed repeatedly in what follows.
  \begin{lemma} \label{estimate on app solution}
    There exists $\epsilon_0>0$ such that for any $\epsilon\in (0, \epsilon_0)$, there holds
    \begin{align} \label{estimate on ua}
        |u^a-u_e(y)|\leq& C\epsilon^{\frac{2}{3}}u_e(y),~~   |\partial_yu^a|\leq C, ~~ |(u^a-u_e(y))_{yy}|\leq C\epsilon^{\frac13}, \nonumber\\
         |\partial_x^j u^a|\leq& C_j \epsilon, ~~ |\partial_x^j u^a|\leq C_j \epsilon^{\frac{2}{3}} u_e(y),~~j\geq 1,
       \end{align}
    and 
    \begin{align} \label{estimate on va}
      |\partial_x^j v^a|\leq C_j \epsilon~~,  |\partial_x^j v^a|\leq C_j \epsilon u_e(y), ~~ |\partial_x^j v_y^a|\leq C_j \epsilon^{\frac{2}{3}} u_e(y), ~~ j\geq 0.
    \end{align}
    where $C,C_j$ are independent of $\epsilon$.
\end{lemma}
\begin{proof}
    For $y\in [0, \frac12]$, by (\ref{estimate on approximate solution}), we have 
    \begin{align} \label{p50}
    \begin{aligned}
        |(u^a-u_e(y))_y|&\leq |(u_e^a-u_e(y))_y|+|\tilde{u}_{p,y}^a|+\epsilon^5|h_y| \\
        &\leq C\epsilon + C\epsilon^{-\frac{1}{3}}|\tilde{u}_{p,\xi}^a|+C\epsilon^{\frac{11}{3}} \leq C\epsilon^{\frac{2}{3}}. 
    \end{aligned}
    \end{align}
    Combining the boundary condition $(u^a-u_e(y))|_{y=0}=0$, we deduce for $y\in [0, \frac12],$
    \begin{align} 
    \begin{aligned}
        |u^a-u_e(y)|\leq  C\epsilon^{\frac{2}{3}}y\leq C\epsilon^{\frac{2}{3}}u_e(y).\nonumber
    \end{aligned}
    \end{align}
    Similarly, for $y\in [\frac12, 1]$ there holds $|u^a-u_e(y)|\leq C\epsilon^{\frac{2}{3}}u_e(y)$. By (\ref{estimate on approximate solution}), we obtain directly 
    $ |\partial_yu^a|\leq C, ~~ |(u^a-u_e(y))_{yy}|\leq C\epsilon^{\frac13}$.\\
    
    For $1 \leq j \in \mathbb{N}$, the same argument as (\ref{p50}) gives $\|\partial_x^j u_y^a\|_{L^{\infty}}\leq C\epsilon^{\frac{2}{3}}$ for $y \in [0,1]$. Thus, by using the boundary condition $\partial_x^ju^a|_{y=0,1}=0$, we have
    \begin{align} 
        |\partial_x^ju^a|\leq C\epsilon^{\frac{2}{3}}u_e(y),~~for ~~ y\in [0, 1].\nonumber
    \end{align}
    
    The first estimate of (\ref{estimate on va}) can be obtained by the same argument and the second one of (\ref{estimate on va}) is a direct consequence of (\ref{estimate on ua}) with the incompressibility.\\
\end{proof}

     Then, we introduce a Lemma which will be useful frequently.
    \begin{lemma} 
    Let $f(y) \in C^1([0,1])$ satisfy $f(\frac{1}{2})=0$, then 
    \begin{align} \label{weight estimate inequation}
    \int_{0}^{1} f^2(y) \mathrm{d}y \le \ln 2 \int_{0}^{1}y(1-y)(f'(y))^2 \mathrm{d}y.
    \end{align}
\end{lemma}
\begin{proof}
Since $f(\frac{1}{2})=0$, we have
\[
f(y)=-\int_{y}^{\frac{1}{2}}f'(t)\mathrm{d}t, ~~ y\in [0,1]
\]
hence
\[
f^2(y)\leq \int_{y}^{\frac{1}{2}}t(1-t)(f'(t))^2\mathrm{d}t\int_{y}^{\frac{1}{2}}\frac{1}{t(1-t)}\mathrm{d}t=\int_{y}^{\frac{1}{2}}t(1-t)(f'(t))^2\mathrm{d}t\ln \frac{1-y}{y}.
\]
Thus, there holds 
\[
\begin{aligned}
\int_{0}^{\frac{1}{2}}f^2(y)\mathrm{d}y&\leq \int_{0}^{\frac{1}{2}}\ln \frac{1-y}{y}\mathrm{d}y\int_{y}^{\frac{1}{2}}t(1-t)(f'(t))^2\mathrm{d}t \\
&=\int_{0}^{\frac{1}{2}}t(1-t)(f'(t))^2\mathrm{d}t\int_{0}^{t}\ln \frac{1-y}{y}\mathrm{d}y \\
&\leq \ln2 \int_{0}^{\frac{1}{2}}y(1-y)(f'(y))^2\mathrm{d}y.
\end{aligned}
\]
Similarly, we obtain
\[
\int_{\frac{1}{2}}^{1}f^2(y)\mathrm{d}y \leq \ln2 \int_{\frac{1}{2}}^{1}y(1-y)(f'(y))^2\mathrm{d}y.
\]
Adding them together, we obtain (\ref{weight estimate inequation}).
\end{proof}

  We now turn to the linear stability estimate of the error equation. For simplicity, we set $\hat{H}(\Omega):= \left \{ f\in H_0^1(\Omega) ~|~ f(x,y)=f(x,1-y), \forall (x, y)\in \Omega \right \}. $
\begin{proposition} \label{pro-linear stability estimate}
    Let $(u, v)\in \hat{H}(\Omega) \times \widetilde{H}(\Omega)$ be a smooth solution of (\ref{error}), then there exists $\epsilon_0 > 0 $ such that for any $\epsilon \in (0, \epsilon_0)$, there holds
    \begin{align}  \label{linear stability estimate}
        \begin{aligned}
            &\int_{\Omega} u_e(y)(u_x^2+v_x^2)\mathrm{d}x\mathrm{d}y+\epsilon\int_{\Omega} (u_{0,y}^2+u_e(y)u_y^2) \mathrm{d}x\mathrm{d}y +\epsilon^2\int_{\Omega} (u_{xx}^2+u_{xy}^2+v_{xx}^2+v_{xy}^2)\mathrm{d}x\mathrm{d}y  \\
            \leq & C\Big|\int_{\Omega}(R_uu_x+R_vv_x)\mathrm{d}x\mathrm{d}y\Big| + C\int_{\Omega}(R_u^2+u_e(y)R_v^2)\mathrm{d}x\mathrm{d}y\\
            &+C\Big|\int_{\Omega}R_uu_0\mathrm{d}x\mathrm{d}y\Big|+C\epsilon\Big|\int_{\Omega}(R_uu_{xx}+R_vv_{xx})\mathrm{d}x\mathrm{d}y\Big|.
        \end{aligned}
    \end{align}
    \end{proposition}
    \begin{proof}The estimate (\ref{linear stability estimate}) can be obtained by taking $\epsilon_0$ small and combining the positivity estimate (\ref{positive estimate}), energy estimate (\ref{energy estimate}) and (\ref{energy estimate2}).
    \end{proof}

\subsection{Positive estimates of error equation (\ref{error})} \label{positive}
\begin{lemma} \label{le4}
    Let $(u, v)\in \hat{H}(\Omega) \times \widetilde{H}(\Omega)$ be a smooth solution of (\ref{error}), then there exists $\epsilon_0 > 0$ such that for any $\epsilon \in (0, \epsilon_0)$, there holds
    \begin{align} \label{positive estimate}
        \int_{\Omega} u_e(y)(u_x^2+v_x^2)\mathrm{d}x\mathrm{d}y \leq& C\epsilon^2\int_{\Omega} [u_{0,y}^2+u_e(y)\bar{u}^2_y]\mathrm{d}x\mathrm{d}y+C\Big|\int_{\Omega} (R_uu_x+R_vv_x)\mathrm{d}x\mathrm{d}y\Big|.
    \end{align}
\end{lemma}
\begin{proof}
\indent Multiplying the first equation in (\ref{error}) by $ u_x $, the second equation in (\ref{error}) by $ v_x $, adding them together and integrating in $ \Omega $, we have
\begin{align} \label{positive1}
    \begin{aligned}
        &\underbrace{-\epsilon \int_{\Omega} (u_{xx}u_x+u_{yy}u_x+v_{xx}v_x+v_{yy}v_x)\mathrm{d}x\mathrm{d}y}_{diffusion\quad term} \\
        +&\underbrace{\int_{\Omega}(\partial_xpu_x+\partial_ypv_x)\mathrm{d}x\mathrm{d}y}_{pressure \quad term}+\underbrace{\int_{\Omega}(S_uu_x+S_vv_x)\mathrm{d}x\mathrm{d}y}_{linear \quad term} \\
        =&\int_{\Omega}(R_uu_x+R_vv_x)\mathrm{d}x\mathrm{d}y.
    \end{aligned}
\end{align}
Then we deal with them term by term. \\
\textbf{The diffusion term:} Integrating by parts, we obtain that
\[
-\epsilon \int_{\Omega} (u_{xx}u_x+u_{yy}u_x+v_{xx}v_x+v_{yy}v_x)\mathrm{d}x\mathrm{d}y=0.
\]
\textbf{The pressure term:} Integrating by parts and using the divergence-free condition, we obtain that
\[
\int_{\Omega}(\partial_xpu_x+\partial_ypv_x)\mathrm{d}x\mathrm{d}y=0.
\]
\textbf{Linear term:} Direct computation gives 
\[
\begin{aligned}
&\int_{\Omega}(S_uu_x+S_vv_x)\mathrm{d}x\mathrm{d}y \\
=&\int_{\Omega}\big( (u^a\partial_xu+v^a\partial_yu+u\partial_xu^a+v\partial_yu^a)u_x+(u^a\partial_xv+v^a\partial_yv+u\partial_xv^a+v\partial_yv^a)v_x\big)\mathrm{d}x\mathrm{d}y \\
=&\underbrace{\int_{\Omega}u^a(u_x^2+v_x^2)\mathrm{d}x\mathrm{d}y}_{I_1}+\underbrace{\int_{\Omega} v^au_yu_x+v^av_yv_x+v_x^auv_x+v_y^avv_x+u_x^auu_x+u_y^avu_x\mathrm{d}x\mathrm{d}y}_{I_2}.
\end{aligned}
\]
By (\ref{estimate on ua}), we obtain 
\begin{align}\label{estimate I1}
I_1\geq (1-C\epsilon^{\frac23})\int_{\Omega}u_e(y)(u_x^2+v_x^2)\mathrm{d}x\mathrm{d}y.
\end{align}
We deal with $I_2$ term by term. 
By (\ref{estimate on va}) and Cauchy inequality, we have
\begin{align}\label{estimate I21}
\Big|\int_{\Omega}v^au_yu_x\mathrm{d}x\mathrm{d}y\Big|\leq &C\epsilon\int_{\Omega}u_e(y)|u_yu_x|\mathrm{d}x\mathrm{d}y\nonumber\\
\leq& C\epsilon^2\int_{\Omega}u_e(y)u_y^2\mathrm{d}x\mathrm{d}y+\frac{1}{100}\int_{\Omega}u_e(y)u_x^2\mathrm{d}x\mathrm{d}y.
\end{align}
Similarly, there holds
\begin{align}\label{estimate I22}
&\Big|\int_{\Omega}(v^av_yv_x+v_y^avv_x)\mathrm{d}x\mathrm{d}y\Big| \nonumber\\
\leq&C\epsilon\int_{\Omega}u_e(y)(v_y^2+v_x^2)\mathrm{d}x\mathrm{d}y+C\epsilon^{\frac{2}{3}}\int_{\Omega}u_e(y)(v^2+v_x^2)\mathrm{d}x\mathrm{d}y\nonumber\\
\leq&C\epsilon^{\frac{2}{3}}\int_{\Omega} u_e(y)(u_x^2+v_x^2)\mathrm{d}x\mathrm{d}y,
\end{align}
where at the last line, we use the $ \rm{Poincar\acute{e}} $ inequality in $x$ variable for $v$ with zero mean. \\
Moreover, 
\begin{align}\label{estimate I23}
\Big|\int_{\Omega}v_x^auv_x\mathrm{d}x\mathrm{d}y\Big|\leq& C\epsilon \int_\Omega u_e(y)|u v_x|\mathrm{d}x\mathrm{d}y \nonumber\\
\leq & \frac{1}{100}\int_{\Omega}u_e(y)v_x^2\mathrm{d}x\mathrm{d}y+C\epsilon^2 \int_\Omega u_e(y)u^2\mathrm{d}x\mathrm{d}y \nonumber\\
\leq & \frac{1}{100}\int_{\Omega}u_e(y)v_x^2\mathrm{d}x\mathrm{d}y+C\epsilon^2 \int_\Omega u_e(y)\bar{u}_x^2\mathrm{d}x\mathrm{d}y+C\epsilon^2 \int_\Omega u_{0,y}^2\mathrm{d}x\mathrm{d}y.
\end{align}

Then, integrating by parts gives
\[
\begin{aligned}
\Big|\int_{\Omega} u_x^auu_x\mathrm{d}x\mathrm{d}y\Big|=\frac{1}{2}\Big|\int_{\Omega} u_{xx}^au^2\mathrm{d}x\mathrm{d}y\Big|=&\frac{1}{2}\Big|\int_{\Omega}u_{xx}^a(u_0^2+2\bar{u}u_0+\bar{u}^2)\mathrm{d}x\mathrm{d}y\Big|\\
=&\frac{1}{2}\Big|\int_{\Omega}u_{xx}^a(2\bar{u}u_0+\bar{u}^2)\mathrm{d}x\mathrm{d}y\Big|.
\end{aligned}
\]
By (\ref{estimate on ua}), Hardy inequality (\ref{hardy2}) and $ \rm{Poincar\acute{e}} $ inequality in $x$ variable, we deduce 
\[
\begin{aligned}
\Big|\int_{\Omega}u_{xx}^a\bar{u}u_0\mathrm{d}x\mathrm{d}y\Big|&\leq C\epsilon\int_{\Omega}|\bar{u}u_0|\mathrm{d}x\mathrm{d}y \\
&\leq C\epsilon \int_{\Omega} \Big|\frac{u_0}{y(1-y)}\Big||y(1-y)\bar{u}|\mathrm{d}x\mathrm{d}y \\
&\leq C\epsilon^2\int_{\Omega}\Big(\Big|\frac{u_0}{y}\Big|^2+\Big|\frac{u_0}{1-y}\Big|^2\Big)\mathrm{d}x\mathrm{d}y+\frac{1}{100}\int_{\Omega}u_e(y)\bar{u}^2\mathrm{d}x\mathrm{d}y\\
&\leq C\epsilon^2\int_{\Omega}u_{0,y}^2\mathrm{d}x\mathrm{d}y+\frac{1}{100}\int_{\Omega}u_e(y)u_x^2\mathrm{d}x\mathrm{d}y.
\end{aligned}
\]
By (\ref{estimate on ua}) and $\rm{Poincar\acute{e}} $ inequality in $x$ variable, we deduce 
\[
\Big|\int_{\Omega}u_{xx}^a\bar{u}^2\mathrm{d}x\mathrm{d}y\Big|\leq C\epsilon^{\frac{2}{3}}\int_{\Omega}u_e(y)\bar{u}^2\mathrm{d}x\mathrm{d}y\leq C\epsilon^{\frac{2}{3}}\int_{\Omega}u_e(y)u_x^2\mathrm{d}x\mathrm{d}y.
\]
Collecting the above estimates, we arrive at 
\begin{align}\label{estimate I24}
\Big|\int_{\Omega} u_x^auu_x\mathrm{d}x\mathrm{d}y\Big|\leq C\epsilon^2\int_{\Omega}u_{0,y}^2\mathrm{d}x\mathrm{d}y+\Big(\frac{1}{100}+C\epsilon^{\frac{2}{3}}\Big)\int_{\Omega}u_e(y)u_x^2\mathrm{d}x\mathrm{d}y. 
\end{align}

Finally, by integration by parts and (\ref{estimate on ua}), there holds
\begin{align}
\Big|\int_{\Omega}u_y^avu_x\mathrm{d}x\mathrm{d}y\Big|=&\frac{1}{2}\Big|\int_{\Omega}u_{yy}^av^2\mathrm{d}x\mathrm{d}y\Big|\nonumber\\
\leq& \frac{1}{2}\int_{\Omega}|(u^a-u_e(y))_{yy}v^2\mathrm{d}x\mathrm{d}y+\frac{1}{2}\int_{\Omega}|u_e''(y)|v^2\mathrm{d}x\mathrm{d}y\nonumber\\
\leq& C\epsilon^{\frac13}\int_{\Omega} v^2\mathrm{d}x\mathrm{d}y+\int_{\Omega}\alpha v^2\mathrm{d}x\mathrm{d}y.\nonumber
\end{align}
By (\ref{weight estimate inequation}), we deduce 
\begin{align}
\int_{\Omega}\alpha v^2\mathrm{d}x\mathrm{d}y\leq \ln 2 \int_{\Omega} u_e(y)v_y^2 \mathrm{d}x\mathrm{d}y.\nonumber
\end{align}
Thus, we obtain
\begin{align}\label{estimate I25}
\Big|\int_{\Omega}u_y^avu_x\mathrm{d}x\mathrm{d}y\Big|\leq\Big(\ln 2+C\epsilon^{\frac13}\Big)\int_{\Omega} u_e(y)u_x^2 \mathrm{d}x\mathrm{d}y.
\end{align}
Collecting the estimate (\ref{estimate I21})-(\ref{estimate I25}), we deduce that
\begin{align}\label{estimate I2}
|I_2|\leq \Big(\ln2+\frac{3}{100}+C\epsilon^{\frac13}\Big)\int_{\Omega} u_e(y)(u_x^2+v_x^2)\mathrm{d}x\mathrm{d}y+C\epsilon^2\int_{\Omega} [u_{0,y}^2+u_e(y)\bar{u}^2_y]\mathrm{d}x\mathrm{d}y.
\end{align}
Thus, by (\ref{positive1}), (\ref{estimate I1}) and (\ref{estimate I2}), we deduce that 
 there exist $\epsilon_0 > 0$ such that for any $\epsilon \in (0, \epsilon_0)$, there holds
\[
\begin{aligned}
\int_{\Omega} u_e(y)(u_x^2+v_x^2)\mathrm{d}x\mathrm{d}y \leq& C\epsilon^2\int_{\Omega} [u_{0,y}^2+u_e(y)\bar{u}^2_y]\mathrm{d}x\mathrm{d}y+C\int_{\Omega} (R_uu_x+R_vv_x)\mathrm{d}x\mathrm{d}y.
 \end{aligned}
\]
This completes the proof of this Lemma.
\end{proof}

\subsection{Energy estimate of error equation (\ref{error})} \label{energy}
\begin{lemma} \label{le5}
    Let $(u, v)\in \hat{H}(\Omega) \times \widetilde{H}(\Omega)$ be a smooth solution of (\ref{error}), then there exists $\epsilon_0 > 0$ such that for any $\epsilon \in (0, \epsilon_0)$, there holds
    \begin{align} \label{energy estimate}
        \epsilon\int_{\Omega}(u_{0,y}^2+u_e(y)\bar{u}_y^2)\mathrm{d}x\mathrm{d}y \leq& C\int_{\Omega}u_e(y)(u_x^2+v_x^2)\mathrm{d}x\mathrm{d}y\nonumber\\
       & +\Big|\int_{\Omega}R_uu_0\mathrm{d}x\mathrm{d}y\Big|+C\int_{\Omega}(R_u^2+u_e(y)R_v^2)\mathrm{d}x\mathrm{d}y.
    \end{align}
    \end{lemma}

\begin{proof} The proof is divided into two parts. The first part deals with $\int_{\Omega}u_{0,y}^2\mathrm{d}x\mathrm{d}y$, and the other part takes care of $\int_{\Omega}u_e(y)\bar{u}_y^2\mathrm{d}x\mathrm{d}y$.\\
{\bf Estimate of $\int_{\Omega}u_{0,y}^2\mathrm{d}x\mathrm{d}y$:}
Multiplying $(\ref{error})_1$ by $u_0$ and integrating on $\Omega$, we obtain that
\[
\begin{aligned}
    -\epsilon\int_{\Omega}\bigtriangleup uu_0\mathrm{d}x\mathrm{d}y+\int_{\Omega}(u^au_x+u_x^au+p_x)u_0\mathrm{d}x\mathrm{d}y+\int_{\Omega}(v^au_y+vu_y^a)u_0\mathrm{d}x\mathrm{d}y
    =\int_{\Omega}R_uu_0\mathrm{d}x\mathrm{d}y.
\end{aligned}
\]

Using the boundary condition $u|_{y=0,1}=u_0|_{y=0,1}=0$, we deduce that
\begin{align}\label{diffusion term}
-\epsilon\int_{\Omega}\bigtriangleup uu_0\mathrm{d}x\mathrm{d}y=-\epsilon\int_{\Omega}(u_{xx}+u_{yy})u_0\mathrm{d}x\mathrm{d}y
=-\epsilon\int_{\Omega}u_{0,yy}u_0\mathrm{d}x\mathrm{d}y=\epsilon\int_{\Omega}u_{0,y}^2\mathrm{d}x\mathrm{d}y.
\end{align}

Integrating by parts in $x$,  it is easy to see that
\begin{align}\label{x-term}
\int_{\Omega}(u^au_x+u_x^au+p_x)u_0\mathrm{d}x\mathrm{d}y=\int_{\Omega}\big( (u^au)_x+p_x \big)u_0\mathrm{d}x\mathrm{d}y=0.
\end{align}

Using the facts that 
\[
\int_{0}^{2\pi}v^a\mathrm{d}x=\int_{0}^{2\pi}v\mathrm{d}x=0,
\]
and $u_0$ being independent on $x$, direct computation gives 
\[
\begin{aligned}
    &\int_{\Omega}(v^au_y+vu_y^a)u_0\mathrm{d}x\mathrm{d}y \\
    =&\int_{\Omega}[v^a(\bar{u}+u_0)_yu_0+v(u^a-u_e(y))_yu_0+vu_e'(y)u_0]\mathrm{d}x\mathrm{d}y \\
    =&\int_{\Omega}[v^a\bar{u}_yu_0+v(u^a-u_e(y))_yu_0]\mathrm{d}x\mathrm{d}y \\
    =&-\int_{\Omega}\bar{u}(v_y^au_0+v^au_{0,y})\mathrm{d}x\mathrm{d}y+\int_{\Omega}v(u^a-u_e(y))_yu_0\mathrm{d}x\mathrm{d}y.
\end{aligned}
\]
Moreover, by (\ref{estimate on va}) and Poincaré inequality in $x$, we deduce 
\[
\begin{aligned}
    \Big|\int_{\Omega}\bar{u}v^au_{0,y}\mathrm{d}x\mathrm{d}y\Big| &\leq C\epsilon\int_{\Omega}u_e(y)|\bar{u}u_{0,y}|\mathrm{d}x\mathrm{d}y \\
    &\leq C\epsilon\int_{\Omega}u_e(y)u_x^2\mathrm{d}x\mathrm{d}y+\frac{\epsilon}{10}\int_{\Omega}u_{0,y}^2\mathrm{d}x\mathrm{d}y
\end{aligned}
\]
and 
\[
\begin{aligned}
   \Big| \int_{\Omega}\bar{u}v_y^au_0\mathrm{d}x\mathrm{d}y \Big|&\leq C\epsilon\Big|\int_{\Omega}\bar{u}u_0\mathrm{d}x\mathrm{d}y\Big| \\
    &=C\epsilon\Big|\int_{\Omega}y(1-y)\bar{u}\Big(\frac{u_0}{y}+\frac{u_0}{1-y}\Big)\mathrm{d}x\mathrm{d}y \Big|\\
    &\leq C\epsilon\int_{\Omega}u_e(y)u_x^2\mathrm{d}x\mathrm{d}y+\frac{\epsilon}{10}\int_{\Omega}u_{0,y}^2\mathrm{d}x\mathrm{d}y,
\end{aligned}
\]
where we use Hardy inequality (\ref{hardy2}).

By (\ref{estimate on approximate solution}), Hardy inequality (\ref{hardy2}) and Poincaré inequality in $x$, we have 
\[
\begin{aligned}
    \Big|\int_{\Omega}v(u^a-u_e(y))_yu_0\mathrm{d}x\mathrm{d}y\Big|&\leq C\epsilon^{\frac23}\int_{\Omega}|v u_0|\mathrm{d}x\mathrm{d}y \\
    &\leq C\epsilon^{\frac23}\Big(\int_{\Omega}\frac{u^2_0}{u^2_e(y)}\mathrm{d}x\mathrm{d}y\Big)^{\frac12}\Big(\int_{\Omega}u_e(y)v^2\mathrm{d}x\mathrm{d}y \Big)^{\frac12}\\
    &\leq C\int_{\Omega}u_e(y)v_x^2\mathrm{d}x\mathrm{d}y+ \frac{\epsilon}{10}\int_{\Omega}u_{0,y}^2\mathrm{d}x\mathrm{d}y.
\end{aligned}
\]
Thus, we obtain 
\begin{align}\label{y-term}
    \Big|\int_{\Omega}(v^au_y+vu_y^a)u_0\mathrm{d}x\mathrm{d}y \Big|\leq C\int_{\Omega}u_e(y)v_x^2\mathrm{d}x\mathrm{d}y+ \frac{\epsilon}{2}\int_{\Omega}u_{0,y}^2\mathrm{d}x\mathrm{d}y.
\end{align}
Collecting (\ref{diffusion term})-(\ref{y-term}), we arrive at 
\begin{align}\label{estimate on zero fre}
\epsilon\int_{\Omega}u_{0,y}^2\mathrm{d}x\mathrm{d}y\leq C\int_{\Omega}u_e(y)(u_x^2+v^2_x)\mathrm{d}x\mathrm{d}y
+\Big|\int_{\Omega}R_uu_0\mathrm{d}x\mathrm{d}y\Big|.
\end{align}

{\bf Estimate of $\int_{\Omega}u_e(y)\bar{u}_y^2\mathrm{d}x\mathrm{d}y$:}
Multiplying $(\ref{vorticity equation})_1$ by $u_e(y)\bar{\phi}$ and integrating on $\Omega$, we obtain 
\[
\begin{aligned}
	\underbrace{\epsilon\int_{\Omega} \triangle^2\phi u_e(y) \bar{\phi}\mathrm{d}x\mathrm{d}y}_{J_0} =& \underbrace{\int_{\Omega}(u^a\triangle\phi_x-\phi_x\triangle u^a)u_e(y)\bar{\phi}\mathrm{d}x\mathrm{d}y}_{J_1}+\underbrace{\int_{\Omega}(v^a\triangle \phi_y -\phi_y \triangle v^a)u_e(y)\bar{\phi}\mathrm{d}x\mathrm{d}y }_{J_2}\\
	&+\underbrace{\int_{\Omega}(\partial_xR_v-\partial_y R_u)u_e(y)\bar{\phi}\mathrm{d}x\mathrm{d}y}_{J_3}.
\end{aligned}
\]
We deal with it term by term. 

\textbf{Term $J_0$:} Due to $\phi = \phi_0+\bar{\phi}$, we have
\[
\begin{aligned}
	J_0=&\epsilon\int_{\Omega}(\bar{\phi}_{xxxx}+2\bar{\phi}_{xxyy}+\bar{\phi}_{yyyy})u_e(y)\bar{\phi}\mathrm{d}x\mathrm{d}y \\
	=&\epsilon\int_{\Omega}\left(u_e(y)\bar{\phi}_{xx}^2+2\bar{\phi}_{xy}(u_e(y)\bar{\phi}_x)_y+\bar{\phi}_{yy}(u_e(y)\bar{\phi})_{yy}\right)\mathrm{d}x\mathrm{d}y \\
	=&\epsilon\int_{\Omega}\left(u_e(y)\bar{\phi}_{xx}^2+2u_e(y)\bar{\phi}_{xy}^2+2u_e'(y)\bar{\phi}_x\bar{\phi}_{xy}+\bar{\phi}_{yy}u_e''(y)\bar{\phi}+2u_e'(y)\bar{\phi}_{y}\bar{\phi}_{yy}+u_e(y)\bar{\phi}_{yy}^2\right)\mathrm{d}x\mathrm{d}y \\
	=&\epsilon\int_{\Omega}\left(u_e(y)\bar{\phi}_{xx}^2+2u_e(y)\bar{\phi}_{xy}^2+u_e(y)\bar{\phi}_{yy}^2-u_e''(y)(\bar{\phi}_x^2+2\bar{\phi}_y^2)\right)\mathrm{d}x\mathrm{d}y.
\end{aligned}
\]
Notice that $u''_e(y)=-2\alpha<0$, we deduce 
\begin{align}\label{estimate of I_0}
    J_0\geq \epsilon\int_{\Omega}u_e(y)\bar{\phi}_{yy}^2\mathrm{d}x\mathrm{d}y=\epsilon\int_{\Omega}u_e(y)\bar{u}_y^2\mathrm{d}x\mathrm{d}y.
\end{align}
\textbf{Term $J_1$:}  Direct computation gives that 
\[
\begin{aligned}
	J_1=& \int_{\Omega}\phi_x[(u^au_e(y)\bar{\phi})_{xx}-u_{xx}^au_e(y)\bar{\phi}]\mathrm{d}x\mathrm{d}y+\int_{\Omega}\phi_x[(u^au_e(y)\bar{\phi})_{yy}-u_{yy}^au_e(y)\bar{\phi}]\mathrm{d}x\mathrm{d}y \\
	=&\underbrace{\int_{\Omega}\bar{\phi}_x[u^au_e(y)\bar{\phi}_{xx}+2u_e(y)u_x^a\bar{\phi}_x]\mathrm{d}x\mathrm{d}y}_{J_{11}}-\underbrace{\int_{\Omega}u^a[(u_e(y)\bar{\phi})_{yy}\bar{\phi}_x+2(u_e(y)\bar{\phi})_y\bar{\phi}_{xy}]\mathrm{d}x\mathrm{d}y}_{J_{12}}.
\end{aligned}
\]
By (\ref{estimate on ua}) and  Poincaré inequality in $x$,  we deduce that
\[
\begin{aligned}
   J_{11} \leq& C\int_{\Omega}|\bar{\phi}_x||(u_e(y)\bar{\phi}_{xx}, u_e(y)\bar{\phi}_x)|\mathrm{d}x\mathrm{d}y\leq C\int_{\Omega}u_e(y)\bar{\phi}_{xx}^2\mathrm{d}x\mathrm{d}y;; \\
   J_{12}\leq &C\int_{\Omega}u_e(y)(|\bar{\phi}\bar{\phi}_x|+|\bar{\phi}_x\bar{\phi}_y|+|\bar{\phi}\bar{\phi}_{xy}|+|\bar{\phi}_y\bar{\phi}_{xy}|)\mathrm{d}x\mathrm{d}y
    +\Big|\int_{\Omega}u^au_e(y)\bar{\phi}_{yy}\bar{\phi}_x\mathrm{d}x\mathrm{d}y\Big| \\
    \leq & C\int_{\Omega}u_e(y)(\bar{\phi}_{xx}^2+\bar{\phi}_{xy}^2)\mathrm{d}x\mathrm{d}y+\Big|\int_{\Omega}u^au_e(y)\bar{\phi}_{yy}\bar{\phi}_x\mathrm{d}x\mathrm{d}y\Big|.
\end{aligned}
\]
Furthermore, by integration by parts, the  (\ref{estimate on ua}) and Cauchy inequality, we have
\[
\begin{aligned}
        &\Big|\int_{\Omega}u^au_e(y)\bar{\phi}_{yy}\bar{\phi}_x\mathrm{d}x\mathrm{d}y\Big| \\
        \leq&\Big|\int_{\Omega}\big((u_y^au_e(y)+u^au_e')\bar{\phi}_y\bar{\phi}_x-\frac{1}{2}u_x^au_e(y)\bar{\phi}_y^2\big)\mathrm{d}x\mathrm{d}y\Big| \\
        \leq&C\int_{\Omega}u_e(y)(\bar{\phi}_y^2+\bar{\phi}_x^2)\mathrm{d}x\mathrm{d}y 
       \leq C\int_{\Omega}u_e(y)(\bar{\phi}_{xy}^2+\bar{\phi}_{xx}^2)\mathrm{d}x\mathrm{d}y.
\end{aligned}
\]
Hence, we obtain 
\begin{align}\label{estimate of I_1}
J_1\leq C\int_{\Omega}u_e(y)(u_x^2+v_x^2)\mathrm{d}x\mathrm{d}y.
\end{align}
\textbf{Term $J_2$:} Direct computation give that 
\[
    \begin{aligned}
      J_2
        =&\int_{\Omega}\phi_y[(v^au_e(y)\bar{\phi})_{xx}-v_{xx}^au_e(y)\bar{\phi}]\mathrm{d}x\mathrm{d}y+\int_{\Omega}\phi_y[(v^au_e(y)\bar{\phi})_{yy}-v_{yy}^au_e(y)\bar{\phi}]\mathrm{d}x\mathrm{d}y\\
        =&\underbrace{\int_{\Omega}\phi_y[v^au_e(y)\bar{\phi}_{xx}+2u_e(y)v_x^a\bar{\phi}_x]\mathrm{d}x\mathrm{d}y}_{J_{21}}-
        \underbrace{\int_{\Omega}v^a[\phi_y(u_e(y)\bar{\phi})_{yy}
        +2\phi_{yy}(u_e(y)\bar{\phi})_y]\mathrm{d}x\mathrm{d}y}_{J_{22}}.
 \end{aligned}
\]
By (\ref{estimate on va}) and  Poincaré inequality in $x$,  we deduce that
  \[
    \begin{aligned}     
       J_{21} \leq& C\epsilon\int_{\Omega}|\phi_y||(u_e(y)\bar{\phi}_{xx},u_e(y)\bar{\phi}_x)\mathrm{d}x\mathrm{d}y
       \leq C\epsilon^2\int_{\Omega}u_e(y)\phi_y^2+C\int_{\Omega}u_e(y)\phi_{xx}^2\mathrm{d}x\mathrm{d}y.
    \end{aligned}
\]
Similarly, we can deduce that 
\[
    \begin{aligned}     
      \Big|\int_{\Omega}v^a\phi_{yy}(u_e(y)\bar{\phi})_y\mathrm{d}x\mathrm{d}y\Big|
       \leq& C\epsilon\int_{\Omega}u_e(y)|\phi_{yy}|(|\bar{\phi}|+|\bar{\phi}_y|)\mathrm{d}x\mathrm{d}y\\
       \leq& C\epsilon^2\int_{\Omega}u_e(y)\phi_{yy}^2+C\int_{\Omega}u_e(y)(\phi_{xx}^2+\phi_{xy}^2)\mathrm{d}x\mathrm{d}y,\\
       \Big|\int_{\Omega}v^a\phi_y(u_e(y)\bar{\phi})_{yy}\mathrm{d}x\mathrm{d}y\Big|\leq & \Big|\int_{\Omega}v^a\phi_yu_e(y)\bar{\phi}_{yy}\mathrm{d}x\mathrm{d}y\Big|
       +\int_{\Omega}|v^a\phi_y|(|\bar{\phi}|+|\bar{\phi}_y|)\mathrm{d}x\mathrm{d}y\\
       \leq & C\epsilon\int_{\Omega}u_e(y)|\phi_y\bar{\phi}_{yy}|\mathrm{d}x\mathrm{d}y+C\epsilon^2\int_{\Omega}u_e(y)\phi_y^2+C\int_{\Omega}u_e(y)(\phi_{xx}^2+\phi_{xy}^2)\mathrm{d}x\mathrm{d}y\\
       \leq & C\epsilon\int_{\Omega}u_e(y)\phi_y^2+\frac{\epsilon}{100}\int_\Omega u_e(y)\bar{\phi}^2_{yy}\mathrm{d}x\mathrm{d}y\nonumber\\
       &+C\int_{\Omega}u_e(y)(\phi_{xx}^2+\phi_{xy}^2)\mathrm{d}x\mathrm{d}y.
    \end{aligned}
\]
Thus, we deduce that 
\begin{align*}
J_{22}\leq C\epsilon\int_{\Omega}u_e(y)\phi_y^2\mathrm{d}x\mathrm{d}y+C\epsilon^2\int_{\Omega}u_e(y)\phi_{yy}^2\mathrm{d}x\mathrm{d}y+\frac{\epsilon}{100}\int_\Omega u_e(y)\bar{\phi}^2_{yy}\mathrm{d}x\mathrm{d}y+C\int_{\Omega}u_e(y)(\phi_{xx}^2+\phi_{xy}^2)\mathrm{d}x\mathrm{d}y.
\end{align*}
Hence, there holds
\begin{align}\label{estimate of I_2}
J_2\leq  C\epsilon\int_{\Omega}u_e(y)(\phi_y^2+\phi_{0,yy}^2)\mathrm{d}x\mathrm{d}y+\Big(\frac{\epsilon}{100}+C\epsilon^2\Big)\int_\Omega u_e(y)\bar{\phi}^2_{yy}\mathrm{d}x\mathrm{d}y+C\int_{\Omega}u_e(y)(\phi_{xx}^2+\phi_{xy}^2)\mathrm{d}x\mathrm{d}y.
\end{align}
\textbf{Term $J_3$:} By (\ref{weight estimate inequation}), integration by parts, Cauchy inequality and Poincaré inequality in $x$, we obtain 
    \begin{align}\label{estimate of I_3}
        J_3=&\int_{\Omega}R_u(u_e(y)\bar{\phi})_y-R_vu_e(y)\bar{\phi}_x\mathrm{d}x\mathrm{d}y \nonumber\\
        \leq&C\int_{\Omega}|R_u||(u_e(y)\bar{\phi}_y, \bar{\phi})|+|R_v||u_e(y)\bar{\phi}_x|\mathrm{d}x\mathrm{d}y \nonumber\\
        \leq&C\int_{\Omega}(u_e(y)\bar{\phi}_y^2+u_e(y)\bar{\phi}_x^2+\bar{\phi}^2)\mathrm{d}x\mathrm{d}y+\int_{\Omega}(R_u^2+u_e(y)R_v^2)\mathrm{d}x\mathrm{d}y\nonumber\\
        \leq&C\int_{\Omega}u_e(y)(\bar{\phi}_{xy}^2+\bar{\phi}_{xx}^2)\mathrm{d}x\mathrm{d}y+\int_{\Omega}(R_u^2+u_e(y)R_v^2)\mathrm{d}x\mathrm{d}y.
    \end{align}
Collecting (\ref{estimate of I_0})-(\ref{estimate of I_3})  and choosing $\epsilon_0$ small, we deduce that for any $\epsilon\in (0,\epsilon_0)$, there holds
\[
    \begin{aligned}
        \epsilon\int_{\Omega}u_e(y)\bar{u}_y^2\mathrm{d}x\mathrm{d}y \leq & C\int_{\Omega}u_e(y)(u_x^2+v_x^2)\mathrm{d}x\mathrm{d}y+C\epsilon\int_{\Omega}u_e(y)u_{0,y}^2\mathrm{d}x\mathrm{d}y\\
         &+C\epsilon\int_{\Omega}u_e(y)\phi_{y}^2\mathrm{d}x\mathrm{d}y+C\int_{\Omega}(R_u^2+u_e(y)R_v^2)\mathrm{d}x\mathrm{d}y.
    \end{aligned}
\]
Moreover, by Poincaré inequality in $x$, we deduce that 
\[
\int_{\Omega}u_e(y)\phi_y^2\mathrm{d}x\mathrm{d}y\leq 2\int_{\Omega} u_e(y)(\bar{\phi}^2_y+\phi^2_{0,y}) \mathrm{d}x\mathrm{d}y \leq C \int_{\Omega}(u_e(y)u_x^2+u_{0,y}^2 )\mathrm{d}x\mathrm{d}y.
\]
Thus, we obtain
\begin{align} \label{estimate on non-zero fre}
        \epsilon\int_{\Omega}u_e(y)\bar{u}_y^2\mathrm{d}x\mathrm{d}y \leq & C\int_{\Omega}u_e(y)(u_x^2+v_x^2)\mathrm{d}x\mathrm{d}y+C\epsilon\int_{\Omega}u_{0,y}^2\mathrm{d}x\mathrm{d}y\nonumber\\
        &+C\int_{\Omega}(R_u^2+u_e(y)R_v^2)\mathrm{d}x\mathrm{d}y.
    \end{align}

Finally, combining the estimate (\ref{estimate on zero fre}) and (\ref{estimate on non-zero fre}), we obtain (\ref{energy estimate}) by  taking $\epsilon_0$ sufficiently small.
\end{proof}

\begin{lemma} \label{energy2}
    Let $(u, v)\in \hat{H}(\Omega) \times \widetilde{H}(\Omega)$ be a smooth solution of (\ref{error}), then there exist $\epsilon_0 > 0$ such that for any $\epsilon \in (0, \epsilon_0)$, there holds
    \begin{align} \label{energy estimate2}
        \epsilon^2\int_{\Omega}(u_{xx}^2+u_{xy}^2+v_{xx}^2+v_{xy}^2)\mathrm{d}x\mathrm{d}y \leq& C\int_{\Omega}u_e(y)(u_x^2+v_x^2)\mathrm{d}x\mathrm{d}y\nonumber+ C\epsilon \int_{\Omega}(u_{0, y}^2+u_e(y)\bar{u}_y^2)\mathrm{d}x\mathrm{d}y\\
       & +C\epsilon\Big|\int_{\Omega}(R_uu_{xx}+R_vv_{xx})\mathrm{d}x\mathrm{d}y\Big|.
    \end{align}
    \end{lemma}

\begin{proof}
\indent Multiplying the first equation in (\ref{error}) by $ \epsilon u_{xx} $, the second equation in (\ref{error}) by $ \epsilon v_{xx} $, adding them together and integrating in $ \Omega $, we have
\begin{align} 
    \begin{aligned}
        &\underbrace{-\epsilon^2 \int_{\Omega} (u_{xx}u_{xx}+u_{yy}u_{xx}+v_{xx}v_{xx}+v_{yy}v_{xx})\mathrm{d}x\mathrm{d}y}_{K_1} \\
        &+\epsilon\underbrace{\int_{\Omega}(\partial_xpu_{xx}+\partial_ypv_{xx})\mathrm{d}x\mathrm{d}y}_{K_2}+\epsilon\underbrace{\int_{\Omega}(S_uu_{xx}+S_vv_{xx})\mathrm{d}x\mathrm{d}y}_{K_3} \\
        =&\epsilon\int_{\Omega}(R_uu_{xx}+R_vv_{xx})\mathrm{d}x\mathrm{d}y.\nonumber
    \end{aligned}
\end{align}
Then we deal with them term by term. \\
\textbf{Term $K_1$:} Integrating by parts, we deduce 
\begin{align}\label{estimate viscocity}
K_1=-\epsilon^2\int_{\Omega} (u_{xx}^2+u_{xy}^2+v_{xx}^2+v_{xy}^2)\mathrm{d}x\mathrm{d}y.
\end{align}
\textbf{Term $K_2$:} Integrating by parts and using the divergence-free condition, we obtain that
\begin{align}\label{estimate pressure}
K_2=0.
\end{align}
\textbf{Term $K_3$:} 
First, there holds
    \[
    \Big|\int_{\Omega} S_{u}u_{xx}\mathrm{d}x\mathrm{d}y\Big|=\Big|\int_{\Omega}[u^au_x+v^au_y+u_x^au+u_y^av]u_{xx}\mathrm{d}x\mathrm{d}y\Big|.
    \]
 By  Lemma (\ref{estimate on app solution}), we deduce that
    \[
    \begin{aligned}
    \Big|\int_{\Omega} u^au_xu_{xx}\mathrm{d}x\mathrm{d}y\Big|=&\frac{1}{2}\Big|\int_{\Omega}u_x^au_x^2\mathrm{d}x\mathrm{d}y\Big| \leq C\epsilon^{\frac{2}{3}}\|\sqrt{u_e}u_x\|_{L^2}^2,\\
      \Big|\int_{\Omega}v^a u_yu_{xx}\mathrm{d}x\mathrm{d}y\Big|=&C\epsilon\int_{\Omega}u_e(y)|u_yu_{xx}|\mathrm{d}x\mathrm{d}y\Big| 
        \leq  C\epsilon\|\sqrt{u_e}u_y\|_{L^2}\|u_{xx}\|_{L^2},\\
        \Big|\int_{\Omega} u_x^auu_{xx}\mathrm{d}x\mathrm{d}y\Big|=&\Big|\int_{\Omega}(u_x^au_0u_{xx}+u_x^a\bar{u}u_{xx})\mathrm{d}x\mathrm{d}y\Big| \\
        \leq & C\epsilon\|u_{0,y}\|_{L^2}\|u_{xx}\|_{L^2}+C\epsilon^{\frac{2}{3}}\|\sqrt{u_e}u_x\|_{L^2}\|u_{xx}\|_{L^2},\\
        \Big|\int_{\Omega} u_y^avu_{xx}\mathrm{d}x\mathrm{d}y\Big|\leq&C\int_{\Omega}|vu_{xx}|\mathrm{d}x\mathrm{d}y
        \leq  C\|\sqrt{u_e}v_y\|_{L^2}\|u_{xx}\|_{L^2}.
    \end{aligned}
    \]
    Thus, there exist $\epsilon_0 > 0$ such that for any $\epsilon \in (0, \epsilon_0)$, there holds
    \begin{align} \label{estimate suuxx}
    \begin{aligned}
        \Big|\int_{\Omega}S_{u}u_{xx}\mathrm{d}x\mathrm{d}y\Big| 
        \leq \frac{\epsilon}{5}\|u_{xx}\|_{L^2}^2+C\epsilon^{-1}\|(\sqrt{u_e}u_x, \sqrt{u_e}v_x, \sqrt{\epsilon u_e}u_y, \sqrt{\epsilon}u_{0,y})\|_{L^2}^2.
    \end{aligned}
    \end{align}
    Similarly, we can obtian
    \begin{align} \label{estimate svvxx}
    \begin{aligned}
        \Big|\int_{\Omega}S_{v}v_{xx}\mathrm{d}x\mathrm{d}y\Big| 
        \leq  \frac{\epsilon}{5}\| v_{xx}\|_{L^2}^2+C\epsilon^{-1}\|(\sqrt{u_e}u_x, \sqrt{u_e}v_x, \sqrt{\epsilon u_e }u_y, \sqrt{\epsilon }u_{0,y})\|_{L^2}^2.
    \end{aligned}
    \end{align}
Combining (\ref{estimate suuxx}) and (\ref{estimate svvxx}), we obtain 
\begin{align}\label{estimate I_3}
|K_3|\leq \frac{2\epsilon^2}{5}\| v_{xx}\|_{L^2}^2+C\|(\sqrt{u_e}u_x, \sqrt{u_e}v_x, \sqrt{\epsilon u_e }u_y, \sqrt{\epsilon }u_{0,y})\|_{L^2}^2.
\end{align}
(\ref{energy estimate2}) can be obtained by collecting (\ref{estimate viscocity}), (\ref{estimate pressure}) and (\ref{estimate I_3}).
\end{proof}

\section{existence of the error equation}\label{sec4}

In this section, we establish the existence of a solution to the error equation (\ref{error equation}). Firstly, we give the following anisotropic Sobolev embedding theorem, its proof can be found in \cite{FGLT2}.
\begin{lemma}
    Let $(u,v)$ be smooth periodic function in $x$ and satisfies $(u,v)|_{y=0,1}=0$, then 
    \begin{align} \label{inequation2}
    \|(u,v)\|_{L^{\infty}}\leq C\big( \|(u_x, v_x)\|_{L^2}+\|(u_y,v_y)\|_{L^2}+\|(u_{xy},v_{xy})\|_{L^2} \big).
    \end{align}
\end{lemma}

\subsection{Existence for the error equations}
\begin{proposition} \label{existence error}
    There exist $\epsilon_0 > 0$ such that for any $\epsilon \in (0, \epsilon_0)$, the error equations (\ref{error})  have a solution $(u, v)$ which satisfies
    \[
    \|(u, v)\|_{L^{\infty}} \le C\epsilon.
    \]
\end{proposition}
\begin{proof}
    For a smooth function $(u,v)$ which satisfies
    \begin{align} \label{smooth function}
        \left \{
        \begin{aligned}
            &u_x+v_y=0, \\
            &(u,v)(x,y)=(u,v)(x+2\pi, y), \\
            &u(x,1)=u(x,0)=v(x,1)=v(x,0)=0, \\
            &u(x, y)=u(x, 1-y),~~v(x, y)=-v(x, 1-y),
        \end{aligned}
        \right .
    \end{align}
    we consider the following linear problem:
    \[
    \left \{
    \begin{aligned}
        &-\epsilon \triangle \tilde{u}+\tilde{p}_x+S_{\tilde{u}}=R_u, \\
        &-\epsilon \triangle \tilde{v}+\tilde{p}_y+S_{\tilde{v}}=R_v, \\
        &\partial_x\tilde{u}+\partial_y\tilde{v} = 0, \\
        &(\tilde{u}, \tilde{v})(x,y)=(\tilde{u}, \tilde{v})(x+2\pi,y),\\
        &\tilde{u}(x,1)=\tilde{v}(x,1)=\tilde{u}(x,0)=\tilde{v}(x,0)=0, \\
        &\tilde{u}(x,y)=\tilde{u}(x,1-y),~\tilde{v}(x,y)=-\tilde{v}(x,1-y),
    \end{aligned}
    \right .
    \]
    where
    \[
    \begin{aligned}
        &S_{\tilde{u}}:=u^a\tilde{u}_x+v^a\tilde{u}_y+\tilde{u}u_x^a+\tilde{v}u_y^a, \\
        &S_{\tilde{v}}:=u^a\tilde{v}_x+v^a\tilde{v}_y+\tilde{u}v_x^a+\tilde{v}v_y^a, \\
        &R_u:=-R_u^a-uu_x-vu_y,~~R_v:=-R_v^a-uv_x-vv_y.
    \end{aligned}
    \]
    By the linear stability estimate  (\ref{linear stability estimate}), we deduce that there exist $\epsilon_0 > 0$ such that for any $\epsilon \in (0, \epsilon_0)$, there holds
    \begin{align} \label{e1}
        \begin{aligned}
            &\|(\sqrt{u_e}\tilde{u}_x, \sqrt{u_e}\tilde{v}_x,\sqrt{\epsilon }\tilde{u}_{0,y},\sqrt{\epsilon u_e}\tilde{u}_y,\epsilon \tilde{u}_{xx}, \epsilon \tilde{u}_{xy}, \epsilon \tilde{v}_{xx}, \epsilon \tilde{v}_{xy})\|_{L^2}^2 \\[3pt]
            \leq &C\Big|\int_{\Omega}(R_u\tilde{u}_x+R_v\tilde{v}_x)\mathrm{d}x\mathrm{d}y\Big|+C\Big|\int_{\Omega}(u_eR_v^2+R_u^2)\mathrm{d}x\mathrm{d}y\Big|\\
            &+C\Big|\int_{\Omega}R_u\tilde{u}_0\mathrm{d}x\mathrm{d}y\Big|+C\epsilon \big | \int_{\Omega} (R_u\tilde{u}_{xx}+R_v\tilde{v}_{xx})\mathrm{d}x\mathrm{d}y\big |.
        \end{aligned}
    \end{align}
Using the Cauchy inequality and Poincaré inequality in $y$ direction, we deduce that 
\begin{align} \label{Ruux+Rvvx}
\begin{aligned}
    \Big|\int_{\Omega}(R_u\tilde{u}_x+R_v\tilde{v}_x)\mathrm{d}x\mathrm{d}y\Big| 
    \le & \frac{\epsilon^2}{10} \int_{\Omega} (\tilde{u}_{xy}^2+\tilde{v}_{xy}^2 ) \mathrm{d}x\mathrm{d}y+C \epsilon^{-2} \int_{\Omega} (R_u^2+R_v^2)\mathrm{d}x\mathrm{d}y,\\
    \Big|\int_{\Omega}R_u\tilde{u}_0\mathrm{d}x\mathrm{d}y\Big|\leq& \frac{\epsilon}{10}\|\tilde{u}_{0,y}\|_{L^2}^2+C\epsilon^{-1}\|R_u\|_{L^2}^2,\\
    \epsilon \Big | \int_{\Omega} (R_u\tilde{u}_{xx}+R_v\tilde{v}_{xx})\mathrm{d}x\mathrm{d}y\Big | \leq& \epsilon^4 \int_{\Omega} (\tilde{u}_{xx}^2+\tilde{v}_{xx}^2)\mathrm{d}x\mathrm{d}y+C\epsilon^{-2} \int_{\Omega} (R_u^2+R_v^2)\mathrm{d}x\mathrm{d}y.
\end{aligned}
\end{align}
Moreover, direct computation gives that
    \begin{align} \label{RuRv}
    \|(R_u, R_v)\|_{L^2}^2 \leq C\|(R_u^a, R_v^a)\|_{L^2}^2+C\|(u,v)\|_{L^{\infty}}^2\|(u_y, u_x, v_x)\|_{L^2}^2.
    \end{align}
    By Poincar$\acute{e}$ inequality in $x$, we have
    \begin{align} \label{uyuxvx}
    \|(u_y, u_x, v_x)\|_{L^2}^2 \leq C\|(u_{0,y}, u_{xy}, u_{xx}, v_{xx})\|_{L^2}^2.
    \end{align}
Insert (\ref{Ruux+Rvvx})-(\ref{uyuxvx}) into (\ref{e1}), we can deduce that
\begin{align} \label{total estimate}
    \begin{aligned}
    &\|(\sqrt{u_e}\tilde{u}_x, \sqrt{u_e}\tilde{v}_x,\sqrt{\epsilon }\tilde{u}_{0,y},\sqrt{\epsilon u_e}\tilde{u}_y, \epsilon \tilde{u}_{xx}, \epsilon \tilde{u}_{xy}, \epsilon \tilde{v}_{xx}, \epsilon \tilde{v}_{xy})\|_{L^2}^2 \\[3pt]
    \leq &  C\epsilon^{-2}\|(R_u^a,R_v^a)\|_{L^2}^2+C\epsilon^{-2}\|(u,v)\|_{L^{\infty}}^2(\|u_{0,y}\|_{L^2}^2+\|(u_{xx},v_{xx},v_{xy},u_{xy})\|_{L^2}^2).
    \end{aligned}
\end{align}
\indent Set the total energy to be
    \[
    \|(u,v)\|_E^2:=\|(\sqrt{u_e}u_x, \sqrt{u_e}v_x,\sqrt{\epsilon u_e}u_y,\sqrt{\epsilon} u_{0,y}, \epsilon u_{xx}, \epsilon v_{xx}, \epsilon v_{xy}, \epsilon u_{xy})\|_{L^2}^2.
    \]
   By (\ref{total estimate}), we deduce 
    \begin{align} 
        \|(\tilde{u}, \tilde{v})\|_E^2 \leq C\epsilon^{-2} \|(R_u^a, R_v^a)\|_{L^2}^2+C\epsilon^{-4}\|(u, v)\|_{L^{\infty}}^2\|(u, v)\|_E^2.\nonumber
    \end{align}   
    By anisotropic Sobolev inequality (\ref{inequation2}) and Poincaré inequality in $x$ direction, we have
    \[
    \|(u,v)\|_{L^{\infty}}^2 \leq C\|u_{0,y}\|_{L^2}^2+C\|(u_{xx}, v_{xx}, u_{xy}, v_{xy})\|_{L^2}^2 \leq C\epsilon^{-2}\|(u,v)\|_E^2.
    \]
    Thus, we have
    \begin{align} \label{e9}
        \|(\tilde{u}, \tilde{v})\|_E^2 \leq C\epsilon^{-2}\|(R_u^a, R_v^a)\|_{L^2}^2+C\epsilon^{-6}\|(u,v)\|_{E}^4.
    \end{align}
    Let $E:=\left \{ (u,v):(u,v) ~ satisfies ~(\ref{smooth function}) ~ and ~ \|(u,v)\|_E < +\infty  \right \}$. Thus, due to
    \[
    \|(R_u^a, R_v^a)\|_{L^2} \leq C\epsilon^5,
    \]
    there exist $\epsilon_0 >0,$ such that for any $\epsilon \in (0, \epsilon_0)$, the operator
    \[
    (u, v) \to (\tilde{u}, \tilde{v})
    \]
    maps the ball $B := \left \{ (u,v) : \|(u,v)\|_E^2 \leq \epsilon^{7} \right \}$ in $B$ into itself. \\
    Moreover, for every two pairs $(u_1, v_1)$ and $(u_2, v_2)$ in the ball, we have
    \begin{align} \label{contraction}
        \|(\tilde{u}_1-\tilde{u}_2, \tilde{v}_1-\tilde{v}_2)\|_E^2 \leq C\epsilon^{-6}\big( \|(u_1, v_1)\|_E^2+\|(u_2, v_2)\|_E^2 \big) \|(u_1-u_2, v_1-v_2)\|_E^2.
    \end{align}
    In fact, set
    \[
    \tilde{U} := \tilde{u}_1-\tilde{u}_2,~~\tilde{V}:=\tilde{v}_1-\tilde{v}_2,~~\tilde{P}=\tilde{p}_1-\tilde{p}_2,
    \]
    then we have
    \[
    \left \{
    \begin{aligned}
        &-\epsilon \triangle\tilde{U}+\tilde{P}_x+S_{\tilde{U}} = R_U, \\
        &-\epsilon \triangle \tilde{V}+\tilde{P}_y+S_{\tilde{V}} = R_V, \\
        &\partial_x\tilde{U}+\partial_y\tilde{V}=0, \\
        &(\tilde{U}, \tilde{V})(x, y) = (\tilde{U}, \tilde{V})(x+2\pi, y), \\
        &\tilde{U}(x, 1) = \tilde{V}(x,1)=\tilde{U}(x, 0) = \tilde{V}(x, 0), \\
        &\tilde{U}(x,y)=\tilde{U}(x,1-y),~\tilde{V}(x,y)=-\tilde{V}(x,1-y),
    \end{aligned}
    \right .
    \]
    where
    \[
    \begin{aligned}
    R_{U} := R_{u_1}-R_{u_2}=&u_2 u_{2,x}+v_2u_{2,y}-u_1u_{1,x}-v_1u_{1,y} \\
    =& (u_2-u_1)\partial_xu_2+u_1\partial_x(u_2-u_1)+(v_2-v_1)\partial_y u_2+v_1\partial_y (u_2-u_1),
    \end{aligned}
    \]
    \[
    \begin{aligned}
    R_{V} := R_{v_1}-R_{v_2}=&u_2 v_{2,x}+v_2v_{2,y}-u_1v_{1,x}-v_1v_{1,y} \\
    =& (u_2-u_1)\partial_xv_2+u_1\partial_x(v_2-v_1)+(v_2-v_1)\partial_y v_2+v_1\partial_y (v_2-v_1).
    \end{aligned}
    \]
    Thus, following the estimate (\ref{e9}) line by line, we obtain (\ref{contraction}). Hence, there exist $\epsilon_0 > 0$ such that for any $\epsilon \in (0, \epsilon_0)$, the operator
    \[
    (u, v) \to (\tilde{u}, \tilde{v})
    \]
    maps the ball $B := \left \{ (u,v) : \|(u,v)\|_E^2 \leq \epsilon^{7} \right \}$ into itself and is a contraction mapping. 
     This completes the proof.
\end{proof}
Finally, we give the proof of Theorem \ref{total theorem}.
\begin{proof}
    Combining the construction of approximate solution (\ref{construction approximate solution}), the estimate (\ref{estimate on approximate solution}) and Proposition \ref{existence error}, we can easily obtian Theorem \ref{total theorem}.
\end{proof}
\section{Appendix}

In this appendix, we give a construction of the corrector $ h(x, y) $ defined in subsection \ref{Approximate solutions}. First, we present a simple lemma, whose proof follows the argument of Lemma 6.1 in \cite{FGLT1}.

\begin{lemma} \label{lemma corrector}
    Assume that $ K(x, y) $ is a $ 2\pi- $periodic smooth function which satisfies 
    \[
    \int_{0}^{2\pi} K(x, y) \mathrm{d}x=0, \quad \forall y\in[0, 1]; \quad K(x,0)=K(x, 1)=0, 
    \]
    then there exist a $ 2\pi$-periodic function $ h(x,y) $ such that
    \begin{align} \label{corrector}
    \begin{aligned}
        &\partial_xh(x, y)=K(x, y);\quad h(x, 0)=h(x, 1)=0; \\
        &\int_{0}^{2\pi} h(x, y)\mathrm{d}x = 0; \quad \| \partial_x^j\partial_y^kh  \|_{L^2}\le C\| \partial_x^j\partial_y^kK \|_{L^2}.
    \end{aligned}
    \end{align}
\end{lemma}
\begin{proof}
    By the Fourier series expansion, we have
    \[
    K(x, y)=\sum\limits_{n \ne 0} K_n(y)e^{inx},~~K_n(0)=K_n(1)=0.
    \]
    Set
    \[
    h(x, y)=\sum\limits_{n \ne 0} \frac{K_n(y)}{in} e^{inx}.
    \]
    It's easy to justify that $h(x, y)$ satisfies (\ref{corrector}) which completes the proof.\\
\end{proof}
Then we construct the corrector $ h(x, y) $ by the above lemma. Direct computation gives
\[
\begin{aligned} 
\partial_x(u_e^a+\tilde{u}_p^a)+\partial_y(v_e^a+\tilde{v}_p^a)=&\chi '(y)(\hat{v}_p^a-v_p^a) \\
=&\epsilon^5\chi'(y)[y^{-15}\gamma^{15}\tilde{v}_p^a-(y-1)^{-15}\xi^{15}v_p^a]=:-\epsilon^5K(x, y).
\end{aligned}
\]
Notice that $ \chi'(y)=0 $, for $ y \in [0, \frac{1}{4}]\cup[\frac{3}{4}, 1] $, and the properties of $ (v_p^a $, $ \hat{v}_p^a) $, we know that $ K(x, y) $ satisfies the assumption in Lemma \ref{lemma corrector}, then there exists $ h(x, y) $ such that
\[
\partial_x(u_e^a+\tilde{u}_p^a)+\partial_y(v_e^a+\tilde{v}_p^a)=-\epsilon^5\partial_xh(x, y),
\]
which indicates that
\[
\partial_xu^a+\partial_yv^a=0.
\]
What's more, since $\chi'(y)=\chi'(1-y)$, we can get $K(x, y)=K(x, 1-y)$, which implies
\[
h(x, y)=h(x, 1-y).
\]

\section*{Acknowledgments}

 T.Tao is partially supported by the NSF of China under Grant 12371234 and 12671282.\\

\textbf{Data availability statement}

We do not analyse or generate any datasets, because our work proceeds within a theoretical and mathematical approach.\\

\textbf{Conflict of interest}

On behalf of all authors, the corresponding author states that there is no conflict of interest.

\end{document}